\documentclass[pdflatex,sn-mathphys-num]{sn-jnl}
\usepackage{graphicx}%
\usepackage{multirow}%
\usepackage{amsmath,amssymb,amsfonts}%
\usepackage{amsthm}%
\usepackage{mathrsfs}%
\usepackage[title]{appendix}%
\usepackage{xcolor}%
\usepackage{textcomp}%
\usepackage{manyfoot}%
\usepackage{booktabs}%
\usepackage{algorithm}%
\usepackage{algorithmicx}%
\usepackage{algpseudocode}%
\usepackage{listings}%
\usepackage{lmodern}
\usepackage{anyfontsize}

\newcommand\blfootnote[1]{%
  \begingroup
  \renewcommand\thefootnote{}\footnote{#1}%
  \addtocounter{footnote}{-1}%
  \endgroup
}

\theoremstyle{definition}

\newtheorem{theorem}{Theorem}[section]
\newtheorem{proposition}[theorem]{Proposition}%
\newtheorem{lemma}[theorem]{Lemma}
\newtheorem{corollary}[theorem]{Corollary}
\newtheorem{example}[theorem]{Example}%
\newtheorem{definition}[theorem]{Definition}%

\begin{document}

\title[Askey-Wilson version of Second Main Theorem for holomorphic curves in projective space]{Askey-Wilson version of Second Main Theorem for holomorphic curves in projective space}

\author*[1]{\fnm{Chengliang} \sur{Tan}}

\author[1]{\fnm{Risto} \sur{Korhonen}}

\affil[1]{\orgdiv{Department of Physics and Mathematics}, \orgname{University of Eastern Finland}, \orgaddress{\street{P.O. Box 111}, \postcode{FI-80101} \city{Joensuu}, \country{Finland}}}

\abstract{In this paper, an Askey-Wilson version of the Wronskian-Casorati determinant $\mathcal{W}(f_{0}, \dots, f_{n})(x)$ for meromorphic functions $f_{0}, \dots, f_{n}$ is introduced to establish an Askey-Wilson version of the general form of the Second Main Theorem in projective space. This improves upon the original Second Main Theorem for the Askey-Wilson operator due to Chiang and Feng. In addition, by taking into account the number of irreducible components of hypersurfaces, an Askey-Wilson version of the Truncated Second Main Theorem for holomorphic curves into projective space with hypersurfaces located in $l$-subgeneral position is obtained.}



\maketitle
\section{Introduction}\label{sec1}
\blfootnote{*Corresponding author.}
\blfootnote{Email address: chengtan@uef.fi (C.L. Tan), risto.korhonen@uef.fi  (R. Korhonen).}
In 1925, Nevanlinna established value distribution theory and presented a Second Main Theorem (SMT) for non-constant meromorphic functions \cite{Nevan}. Subsequently, Cartan extended the SMT to holomorphic curves $f:\mathbb{C}\rightarrow \mathbb{P}^{n}(\mathbb{C})$ \cite{Cartan}. Cartan's theorem can yield better results for certain kinds of problems compared with the SMT \cite{strenth}. In 1997, motivated by Schmidt's Subspace Theorem in number theory, Vojta slightly generalized Cartan's theorem and provided a general form of the SMT \cite{Vojta}. In the same year, by utilizing Ye's techniques \cite{Ye}, Ru established a general form of the SMT with an improved error term, enhancing Vojta's results \cite{Ru1}. Leveraging the general form of the SMT, Ru derived a defect relation for holomorphic curves intersecting hypersurfaces, proving the Shiffman conjecture, which was proposed by B. Shiffman in 1979 \cite{Shiff}. Shortly afterward, Ru also extended the defect relation to algebraic varieties (not limited to projective space) \cite{Ru4}. In 2019, Quang refined Ru’s results \cite{Ru4} to apply to hypersurfaces in subgeneral position \cite{Quang}. Recently, by taking into account the number of irreducible components of hypersurfaces, a new SMT for holomorphic maps into projective space with hypersurfaces was introduced by Shi and Yan \cite{Shi}.

Halburd and the second author were the first to formulate Nevanlinna theory for the difference operator \cite{Ha1, Ha2}, which was successfully applied to detect solutions of various difference equations \cite{Ha3, Ha4}. Subsequently, Halburd, the second author, and Tohge \cite{hyperplane}, as well as Pit-Mann Wong, Hui-Fai Law, and Philip P. W. Wong \cite{Wong}, independently extended difference Nevanlinna theory to projective space. Later, a difference analogue of the SMT for meromorphic mappings into algebraic varieties was provided by P.-C. Hu and N. V. Thin \cite{Hu}. Building on the strengths of the difference analogue of Nevanlinna theory, Cheng and Chiang, as well as Chiang and Feng separately developed Nevanlinna theory for the Wilson operator \cite{NANN} and the Askey-Wilson divided difference operator \cite{NADV}. This development led to the definition of an Askey-Wilson type Nevanlinna deficiency, offering a new interpretation that many significant infinite products arising from the study of basic hypergeometric series should be regarded as zero/pole scarce.

The purpose of this paper is to extend the Askey-Wilson analogue of the SMT for holomorphic curves into projective space with hypersurfaces. In the following section, we begin with some foundational concepts about the Askey-Wilson operator and Nevanlinna characteristics. In the third section, we further generalize the Logarithmic Derivative Lemma for Askey-Wilson operator and the Chiang and Feng's estimation to Logarithmic Derivative. Subsequently, we introduce the Askey-Wilson version of the Wronskian-Casorati determinant $\mathcal{W}(f_{0}, \dots, f_{n})(x)$ for meromorphic functions $f_{0}, \dots, f_{n}$. It is worth noting that it has alternative definitions when $|x|$ is sufficiently large, which will help us address various situations. In Section \ref{sec5}, we present a general form of the Askey-Wilson version (Truncated) SMT into projective space for hyperplanes. Finally, we establish the Askey-Wilson version of the Truncated SMT for hypersurfaces by improving a key lemma due to Shi and Yan. The improved lemma can also be applied to enhance the classical SMT and Schmidt's Subspace Type Theorem.

\section{Askey-Wilson operator and Nevanlinna characteristic}\label{sec2}

Let $f$ be a complex function, $q \in \mathbb{C}$, and $0 < |q| < 1$. For each $x \in \mathbb{C}$, set $x = \frac{z + z^{-1}}{2}$, where we assume $|z| \geq 1$ and $z = x + i\sqrt{1 - x^{2}}$ when $|z| = 1$, so that $x = \text{Re}(z) \in [-1, 1]$. Thus, there corresponds a unique $z$ for each fixed $x \in \mathbb{C}$, with $|z| \rightarrow \infty$ as $|x| \rightarrow \infty$, and vice versa. The \textit{one-sided Askey-Wilson shift operators} $(\eta_{q}f)(x)$ and $(\eta_{q}^{-1}f)(x)$ are defined by 
\begin{equation*}
    (\eta_{q}f)(x) := f\left(\frac{q^{\frac{1}{2}}z + q^{-\frac{1}{2}}z^{-1}}{2}\right), \quad (\eta_{q}^{-1}f)(x) := f\left(\frac{q^{-\frac{1}{2}}z + q^{\frac{1}{2}}z^{-1}}{2}\right).
\end{equation*}
We also write $(\eta_{q}^{0}f)(x) := f(x)$, $(\eta_{q}^{n}f)(x) := (\eta_{q}(\eta_{q}^{n-1}f))(x)$, and $(\eta_{q}^{-n}f)(x) := (\eta_{q}^{-1}(\eta_{q}^{-(n-1)}f))(x)$ for $n \in \mathbb{N}$. It is easy to check that when $|x| > \mathcal{R}(n, q)$ for some positive constants $\mathcal{R}(n, q)$ (specifically, $|z| > |q|^{-\frac{n}{2}}$), we have
\begin{eqnarray}\label{equa2.1}
    (\eta_{q}^{n}f)(x) = f\left(\frac{q^{\frac{n}{2}}z + q^{-\frac{n}{2}}z^{-1}}{2}\right), \quad (\eta_{q}^{-n}f)(x) = f\left(\frac{q^{-\frac{n}{2}}z + q^{\frac{n}{2}}z^{-1}}{2}\right).
\end{eqnarray}
Furthermore, let $f(x)$ be an entire function; then $(\eta_{q}^{n}f)(x)$ and $(\eta_{q}^{-n}f)(x)$ are analytic, and $(\eta_{q}^{-m_{1}}(\eta_{q}^{m_{2}}f))(x) = (\eta_{q}^{m_{2}}(\eta_{q}^{-m_{1}}f))(x) = (\eta_{q}^{m_{2} - m_{1}}f)(x)$ for $m_{1}, m_{2} = 0, \dots, n$ and any $x \in \mathbb{C}$ satisfying $|x| > \mathcal{R}(n, q)$. Generally speaking, these properties for $\eta_{q}^{n}$ and $\eta_{q}^{-n}$ may not hold for any $x \in \mathbb{C}$. For instance, let $I(x) = x$ be the identity map, and set $x_{0} = (q^{-1/4} + q^{1/4})/2$ with the corresponding $z_{0} = q^{-1/4}$; then $(\eta_{q}^{-1}(\eta_{q}I))(x_{0}) = (q^{-3/4} + q^{3/4})/2 \neq I(x_{0})$.

The \textit{Askey-Wilson divided difference operator} $\mathcal{D}_{q}$, which acts on $f$, is defined by
\begin{equation*}
    (\mathcal{D}_{q}f)(x) := \frac{(\eta_{q}f)(x) - (\eta_{q}^{-1}f)(x)}{\eta_{q}x - \eta_{q}^{-1}x},
\end{equation*}
where $x \neq \pm 1$, and we simply write $\eta_{q}^{\pm}x = (\eta_{q}^{\pm}I)(x)$. From now on, we refer to $\mathcal{D}_{q}$ as the Askey-Wilson operator or AW operator for short. We also adopt the notation $(\mathcal{D}_{q}^{0}f)(x) = f(x)$ and $(\mathcal{D}_{q}^{n}f)(x) = (\mathcal{D}_{q}(\mathcal{D}_{q}^{n-1}f))(x)$ for $n \in \mathbb{N}$. In the case $x = \pm 1$, $(\mathcal{D}_{q}f)(\pm 1)$ is given by
\begin{equation*}
    (\mathcal{D}_{q}f)(\pm 1) := \lim\limits_{x \rightarrow \pm 1}(\mathcal{D}_{q}f)(x) = f'\left(\pm \frac{q^{\frac{1}{2}} + q^{-\frac{1}{2}}}{2}\right),
\end{equation*}
when $f$ is differentiable at $\pm (q^{\frac{1}{2}} + q^{-\frac{1}{2}})/2$. According to this definition, for any two complex functions $f$ and $g \not\equiv 0$, we obtain the following product rule and quotient rule for the Askey-Wilson operator:
\begin{eqnarray*}
    (\mathcal{D}_{q}(f g))(x) &=& (\eta_{q}f)(x)(\mathcal{D}_{q}g)(x) + (\eta_{q}^{-1}g)(x)(\mathcal{D}_{q}f)(x)\\
    &=& (\eta_{q}g)(x)(\mathcal{D}_{q}f)(x) + (\eta_{q}^{-1}f)(x)(\mathcal{D}_{q}g)(x),\\
    \left(\mathcal{D}_{q}\frac{f}{g}\right)(x) &=& \frac{(\eta_{q}g)(x)(\mathcal{D}_{q}f)(x) - (\eta_{q}f)(x)(\mathcal{D}_{q}g)(x)}{(\eta_{q}g)(x)(\eta_{q}^{-1}g)(x)}\\
    &=& \frac{(\eta_{q}^{-1}g)(x)(\mathcal{D}_{q}f)(x) - (\eta_{q}^{-1}f)(x)(\mathcal{D}_{q}g)(x)}{(\eta_{q}g)(x)(\eta_{q}^{-1}g)(x)}.
\end{eqnarray*}
Conventionally, the \textit{Askey-Wilson averaging operator} $\mathcal{A}_{q}$ is defined by
\begin{equation*}
    (\mathcal{A}_{q}f)(x) := \frac{(\eta_{q}f)(x) + (\eta_{q}^{-1}f)(x)}{2}.
\end{equation*}
We also adopt the notation
\begin{eqnarray*}
    (\mathcal{A}_{q^{n}}f)(x) &=& \frac{(\eta_{q^{n}}f)(x) + (\eta_{q^{-n}}f)(x)}{2}\\
    &=:& \frac{f\left(\frac{q^{n/2}z + q^{-n/2}z^{-1}}{2}\right) + f\left(\frac{q^{-n/2}z + q^{n/2}z^{-1}}{2}\right)}{2}
\end{eqnarray*}
for $n \in \mathbb{N}$. It follows from (\ref{equa2.1}) that $(\mathcal{A}_{q^{n}}f)(x) = \frac{(\eta_{q}^{n}f)(x) + (\eta_{q}^{-n}f)(x)}{2}$ when $|x| > \mathcal{R}(n, q)$.

With the help of $\mathcal{A}_{q}$, the above product and quotient rules can be rewritten as follows:

\smallskip

\begin{theorem}\label{the2.1}
    (Askey-Wilson product and quotient rules) \cite{NADV} For any pair of complex functions $f$ and $g$, we have:
    
    (i) \[(\mathcal{D}_{q}(f g))(x) = (\mathcal{A}_{q}f)(x)(\mathcal{D}_{q}g)(x) + (\mathcal{A}_{q}g)(x)(\mathcal{D}_{q}f)(x),\]
    
    (ii) \[\left(\mathcal{D}_{q}\frac{f}{g}\right)(x) = \frac{(\mathcal{A}_{q}g)(x)(\mathcal{D}_{q}f)(x) - (\mathcal{A}_{q}f)(x)(\mathcal{D}_{q}g)(x)}{(\eta_{q}g)(x)(\eta_{q}^{-1}g)(x)},\]
    
    whenever $g \not\equiv 0$.
\end{theorem} 

\smallskip

Furthermore, Chiang and Feng proved the following properties for the operators $\mathcal{D}_{q}$ and $\mathcal{A}_{q}$.

\smallskip

\begin{theorem}\label{the2.3}
    \cite[Theorem 2.1]{NADV} Let $f$ be an entire function. Then both $\mathcal{A}_{q}f$ and $\mathcal{D}_{q}f$ are entire. Moreover, if $f(x)$ is meromorphic, then $\mathcal{A}_{q}f$ and $\mathcal{D}_{q}f$ are also meromorphic.
\end{theorem}

\smallskip

By using the composition properties of entire functions or by replacing $q^{1/2}$ and $q^{-1/2}$ with $q^{n/2}$ and $q^{-n/2}$ in the proof of \cite[Theorem 2.1]{NADV}, we can also prove the following similar proposition.

\smallskip

\begin{proposition}\label{pro2.3}
    Let $n \in \mathbb{N}$ and let $f$ be a complex function. If $f$ is entire (meromorphic), then $\mathcal{D}_{q}^{n}f$, $\mathcal{A}_{q^{n}}f$ and $(\eta_{q^{n}}f)(\eta_{q^{-n}}f)$ are all entire (meromorphic).
\end{proposition}

\smallskip

We refer readers to \cite{NANN, NADV, leibniz, Quantum, Qt1, Qt2, Qt3} for additional properties of the AW operator. Now, we review some key notions and results of classical characteristic functions from \cite{strenth, hayman, hyperplane, hyperbolic, Ru2}. Let $n \in \mathbb{N}$. For any non-constant holomorphic curve $f: \mathbb{C} \rightarrow \mathbb{P}^{n}(\mathbb{C})$ with a reduced representation $\textbf{f} = (f_{0}, f_{1}, \dots, f_{n})$, the \textit{characteristic function} $T_{f}(r)$ of $f$ is defined by
\begin{equation*}
    T_{f}(r) := \int_{0}^{2\pi} \frac{\log \lVert \textbf{f}(re^{i\theta}) \rVert}{2\pi} \, d\theta, \quad \lVert \textbf{f}(z) \rVert = \max\limits_{k \in \{0, \dots, n\}} |f_{k}(z)|.
\end{equation*}

It should be noted that $T_{f}(r)$ is independent, up to an additive constant, of the choice of the reduced representation of $f$. We have the following lemma:

\smallskip

\begin{lemma}\label{le2.4}(\cite[Lemma 3.1]{hyperbolic})
    Let $f = (f_{0}, \dots, f_{n})$ be an entire function without common zeros. Then
    \begin{equation}\label{equa2.2}
        N\left(r, \frac{1}{f_{i}}\right) \leq T_{f}(r) + O(1)
    \end{equation}
    as $r \rightarrow \infty$, and
    \begin{equation}\label{equa2.3}
        T\left(r, \frac{f_{i}}{f_{j}}\right) \leq T_{f}(r) + O(1)
    \end{equation}
    for any $i, j = 0, \dots, n$ with $i \neq j$.
\end{lemma}

\smallskip

Suppose $D$ is a hypersurface of degree $d$ in $\mathbb{P}^{n}(\mathbb{C})$ and $Q$ is a homogeneous polynomial of degree $d$ defining $D$, i.e., $D = \{ Q = 0 \}$. The \textit{proximity function} $m_{f}(r, D)$ of $f$ with respect to $D$ is defined by
\begin{equation*}
    m_{f}(r, D) = \int_{0}^{2\pi} \frac{\lambda_{D}\left(\textbf{f}(re^{i\theta})\right)}{2\pi} \, d\theta, \quad \lambda_{D}(\textbf{f}(z)) = \log \frac{\lVert \textbf{f}(z) \rVert^{d}\lVert Q \rVert}{|Q(\textbf{f})(z)|},
\end{equation*}
where $z \not\in f^{-1}(D)$ and $\lVert Q \rVert$ is the maximum of the absolute values of the coefficients in $Q$. The \textit{counting function} $n_{f}(t, D)$ is defined as the number of zeros of $Q(\textbf{f})$ in $|z| < t$, counting multiplicities. The \textit{integrated counting function} is then defined by
\begin{equation*}
    N_{f}(r, D) = \int_{0}^{r} \frac{n_{f}(t, D) - n_{f}(0, D)}{t} \, dt + n_{f}(r, D) \log r.
\end{equation*}

For convenience, we sometimes write $m_{f}(r, D)$ and $N_{f}(r, D)$ as $m_{f}(r, Q)$ and $N_{f}(r, Q)$ if there is no ambiguity. From the Poisson-Jensen formula, we obtain the following theorem:

\smallskip

\begin{theorem}[First Main Theorem]
    Let $f: \mathbb{C} \rightarrow \mathbb{P}^{n}(\mathbb{C})$ be a holomorphic map, and let $D$ be a hypersurface in $\mathbb{P}^{n}(\mathbb{C})$ of degree $d$. If $f(\mathbb{C}) \not\subset \text{Supp} \, D$, then for every finite positive real number $r$,
    \begin{equation*}
        m_{f}(r, D) + N_{f}(r, D) = d T_{f}(r) + O(1),
    \end{equation*}
    where $O(1)$ is a constant independent of $r$.
\end{theorem}

\smallskip

Let $V\subset \mathbb{P}^{n}(\mathbb{C})$ be a smooth complex projective variety of dimension $k\geq 1$, $p, l\in\mathbb{N}$ and $p\geq l$. Hypersurfaces $\{D_{i}\}_{i=1}^{p}$ in $\mathbb{P}^{n}(\mathbb{C})$ are said to be located in \textit{$l$-subgeneral position in $V$} if for any subset $E\subset \{1, \cdots, p\}$ with $\#E=l+1$,
\begin{eqnarray*}
    \bigcap_{j\in E}D_{j}\cap V=\varnothing.
\end{eqnarray*}
If $l=k$, we also say that they are \textit{in general position in $V$}.

\section{Logarithmic hyper-order}\label{sec3}
Suppose $f$ is a meromorphic function. The \textit{logarithmic hyper-order} of $f$ is denoted by
\begin{eqnarray*}
    \sigma_{2}^{\log}(f) := \limsup\limits_{r \rightarrow \infty} \frac{\log^{+} \log^{+} T(r, f)}{\log \log r},
\end{eqnarray*}
where $\log^{+} A = \max\{0, \log A\}$ for $A \geq 0$. Asymptotic relations for a positive, non-decreasing continuous function are frequently used in the construction and application of Nevanlinna theory. For instance, Halburd, the second author, and Tohge considered the asymptotic relation of the real function $T(r)$ with hyper-order strictly less than $1$ \cite{hyperplane}. Zheng and the second author further studied the asymptotic relation of the real function $T(r)$, which is convex with respect to $\log r$ and satisfies $\limsup\limits_{r \rightarrow +\infty} T(r) / \log r = 0$ \cite{log}. Some asymptotic relations for functions of zero order can also be found in \cite{pty}. Now, we consider the case where the logarithmic hyper-order is strictly less than $1$, with the basic idea originating from \cite{hyperplane}.

\smallskip

\begin{lemma}\label{le3.1}
    Let $T(r):[e, +\infty)\rightarrow [0, +\infty)$ be a non-decreasing continuous function. If the logarithmic hyper-order of $T$ is strictly less than one, i.e.,
    \begin{eqnarray}\label{aeq3.1}
        \limsup\limits_{r\rightarrow +\infty}\frac{\log^{+}\log^{+} T(r)}{\log\log r}:=\sigma_{2}^{\log}<1,
    \end{eqnarray}
    and $\delta\in(0, 1-\sigma_{2}^{\log})$, then
    \begin{eqnarray}\label{aeq3.2}
        T(r\log r)=T(r)+o\left(\frac{T(r)}{(\log r)^{\delta}}\right),
    \end{eqnarray}
    where $r$ runs to infinity outside of a set $F$ satisfying
    \begin{eqnarray}\label{aeq3.3}
        \int_{F\cap [e, +\infty)}\frac{dt}{t\log t}<+\infty.
    \end{eqnarray}
\end{lemma}

\smallskip

\begin{proof}
    Let $\tilde{\delta}\in(\delta, 1-\sigma_{2}^{\log})$, $\alpha\in\mathbb{R}^{+}$ and define
    \begin{eqnarray*}
        F_{\alpha}:=\left\{r\in[e, +\infty): \frac{T(r\log r)-T(r)}{T(r)}\cdot (\log r)^{\tilde{\delta}}\geq \alpha\right\}.
    \end{eqnarray*}
    Assume $\int_{F_{\alpha}}1/(t\log t)dt=+\infty$, now we define a sequence $\{r_{n}\}$ using induction as follows. Since $F_{\alpha}$ is closed, we can set $r_{0}=\min\{F_{\alpha}\}$ and define $r_{n}=\min\{F_{\alpha}\cap [r_{n-1}\log r_{n-1}, +\infty)\}$ for all $n\in\mathbb{N}$. Thus, $r_{n+1}\geq r_{n}\log r_{n}$ and $F\subseteq \cup_{n=0}^{\infty}[r_{n}, r_{n}\log r_{n}]$, which leads to $T(r_{n+1})\geq T(r_{n}\log r_{n})\geq  \left(1+\frac{\alpha}{(\log r_{n})^{\tilde{\delta}}}\right)T(r_{n})$ for all $n\in\mathbb{N}_{0}$. Iterating this relation recursively yields
    \begin{equation*}
        T(r_{n})\geq \prod_{k=0}^{n-1}\left(1+\frac{\alpha}{(\log r_{k})^{\tilde{\delta}}}\right)T(r_{0}).
    \end{equation*}
    Assume that there exists a constant $\epsilon >0$ and an integer $m>e$ such that $r_{n}\geq e^{n^{1+\epsilon}}$ for all $n>m$. Then we can find $m_{0}>m$ such that $\log n\leq n^{\epsilon/2}$ for all $n\geq m_{0}$. Since $\log\left(1+\frac{\log\log t}{\log t}\right)$ is non-increasing when $t\geq e^{e}$, then we obtain
    \begin{eqnarray*}
        \int_{F_{\alpha}}\frac{dt}{t\log t}&\leq& \sum_{n=m_{0}}^{+\infty}\int_{r_{n}}^{r_{n}\log r_{n}}\frac{dt}{t\log t}+O(1)\\
        &=&\sum_{n=m_{0}}^{+\infty}\log\left(1+\frac{\log\log r_{n}}{\log r_{n}}\right)+O(1)\\
        &\leq&\sum_{n=m_{0}}^{+\infty}\frac{(1+\epsilon)\log n}{n^{1+\epsilon}}+O(1)\\
        &\leq&\sum_{n=m_{0}}^{+\infty}\frac{1+\epsilon}{n^{1+\epsilon/2}}+O(1)<+\infty,
    \end{eqnarray*}
    which contradicts the assumption. Thus there exists a subsequence $\{r_{n_{j}}\}_{j=1}^{+\infty}\subset \{r_{n}\}$ such that $r_{n_{j}}<e^{n^{1+\epsilon}}$ for all $j\in\mathbb{N}$. Then
    \begin{eqnarray*}
        \limsup\limits_{r\rightarrow +\infty}\frac{\log^{+}\log^{+} T(r)}{\log\log r}&\geq& \limsup\limits_{j\rightarrow +\infty}\frac{\log\left(\log T(r_{0})+\sum_{k=0}^{n_{j}-1}\log\left(1+\frac{\alpha}{(\log r_{k})^{\tilde{\delta}}}\right)\right)}{\log\log r_{n_{j}}}\\
        &\geq&\limsup\limits_{j\rightarrow +\infty}\frac{\log\left(n_{j}\frac{\alpha}{n_{j}^{(1+\epsilon)\tilde{\delta}}}\log\left(1+\frac{\alpha}{n_{j}^{(1+\epsilon)\tilde{\delta}}}\right)^{\frac{n_{j}^{(1+\epsilon)\tilde{\delta}}}{\alpha}}\right)}{(1+\epsilon)\log n_{j}}\\
        &\geq&\limsup\limits_{j\rightarrow +\infty}\frac{(1-(1+\epsilon)\tilde{\delta})\log n_{j}}{(1+\epsilon)\log n_{j}}\geq \frac{1}{1+\epsilon}-\tilde{\delta}.
    \end{eqnarray*}
    It follows that $\limsup\limits_{r\rightarrow+\infty}\frac{\log^{+}\log^{+} T(r)}{\log\log r}\geq 1-\tilde{\delta}$ because $\epsilon$ is arbitrary, which contradicts the assumption (\ref{aeq3.1}). Hence, the proof is complete.
    
\end{proof}

Following the argument of Chiang and Feng in \cite{NADV}, $\sigma_{2}^{\log}(f) < 1$ also implies that $f$ has order zero. Then, by applying a result from Miles \cite{Mile}, we have that $f$ can be represented as a quotient $f = g / h$, where both $g$ and $h$ are entire and $\sigma_{2}^{\log}(g) < 1$ and $\sigma_{2}^{\log}(h) < 1$ hold simultaneously. In this case, the Askey-Wilson operator is well-defined on the class of slow-growing meromorphic functions with logarithmic hyper-order strictly less than $1$. By a similar analysis, we can replace the assumption that $f(x)$ is of finite logarithmic order with $\sigma_{2}^{\log}(f) < 1$ in \cite[Lemma 4.2]{NADV}, thereby obtaining another version of the logarithmic derivative lemma for the AW operator as follows.

\smallskip

\begin{lemma}\label{le3.2}
    Let $n\in\mathbb{N}$ and $f(x)$ be a non-constant meromorphic function such that $\sigma_{2}^{\log}(f)<1$. Then, we have
    \begin{eqnarray}
        m\left(r, \frac{(\mathcal{D}_{q}^{n}f)(x)}{f(x)}\right)&=&O\left(\frac{T(r, f)}{\log\log r}\right),\label{aeq3.2.1}\\
        m\left(r, \frac{(\mathcal{A}_{q^{n}}f)(x)}{f(x)}\right)&=&O\left(\frac{T(r, f)}{\log\log r}\right)\label{aeq3.2.1.1}
    \end{eqnarray}
    and
    \begin{eqnarray}\label{aeq3.2.2}
        m\left(r, \frac{(\eta_{q^{n}}f)(\eta_{q^{-n}}f)(x)}{f^{2}(x)}\right)&=&O\left(\frac{T(r, f)}{\log\log r}\right),
    \end{eqnarray}
    where $r$ runs to infinity outside of a set $F$ satisfying
    \begin{eqnarray}\label{eq3.3}
        \int_{F}\frac{dt}{t\log t}<\infty.
    \end{eqnarray}
\end{lemma}

\smallskip

\begin{proof}
    We only prove (\ref{aeq3.2.1}) here, since a similar proof can be applied for (\ref{aeq3.2.1.1}) and (\ref{aeq3.2.2}) and $(\eta_{q^{\pm n}}f)(x)=(\eta_{q}^{\pm n}f)(x)$ when $|x|$ is sufficiently large. For any $R>e$, we have
     \begin{eqnarray}\label{aeq3.2.3}
        &&N(R\log R, f)\geq \int_{R}^{R\log R}\frac{n(t, f)-n(0, f)}{t}dt+n(0, f)\log (R\log R)\nonumber\\
        &\geq&n(R, f)\int_{R}^{R\log R}\frac{1}{t}dt-n(0, f)\int_{R}^{R\log R}\frac{1}{t}dt+n(0, f)\log (R\log R)\nonumber\\
        &\geq& n(R, f)\log\log R.
     \end{eqnarray}
    Then it follows by substituting $R=r\log r$ into the above inequality and applying Lemma \ref{le3.1} to $N(r, f)$ that
    \begin{eqnarray}\label{aeq3.2.4}
        n(r\log r, f)&\leq& \frac{N((r\log r)\log (r\log r), f)}{\log\log (r\log r)}=O\left(\frac{N(r\log r, f)}{\log\log r}\right)\nonumber\\
        &=&O\left(\frac{N(r, f)}{\log\log r}\right)\leq O\left(\frac{T(r, f)}{\log\log r}\right)
    \end{eqnarray}
    where $r$ runs to infinity outside of a set $F$ satisfying (\ref{eq3.3}). Similarly, we have
    \begin{eqnarray}\label{aeq3.2.5}
        n\left(r\log r, \frac{1}{f}\right)= O\left(\frac{T(r, f)}{\log\log r}\right).
    \end{eqnarray}

    Let integer $k_{1}$ satisfy $-n\leq k_{1}\leq n$. As in inequalities (48), (49), and (84) in \cite{NADV}, we replace $q$ and $q^{-1}$ with $q^{k_{1}}$ and $q^{-k_{1}}$, respectively. Then a form of logarithmic derivative estimation for the operator $\eta_{q}$ is derived as follows:
\begin{eqnarray}\label{aeq3.2.6}
    m\left(r, \frac{(\eta_{q}^{k_{1}} f)(x)}{f(x)}\right) &=& m\left(r, \frac{f[(q^{k_{1}/2} z + q^{-k_{1}/2} z^{-1}) / 2]}{f(x)}\right)\nonumber\\
    &=& O\left(\frac{m(r \log r, f) + m(r \log r, 1/f)}{\log r}\right)\nonumber\\
    &+& O\left(n(r \log r, f) + n(r \log r, 1/f)\right) + O(1)
\end{eqnarray}
for all sufficiently large $|x| = r > 0$. Thus, combining (\ref{aeq3.2.4}), (\ref{aeq3.2.5}), and (\ref{aeq3.2.6}), and applying the First Main Theorem and Lemma \ref{le3.1} to $T(r, f)$, we deduce that
\begin{eqnarray}\label{aeq3.2.7}
    m\left(r, \frac{(\eta_{q}^{k_{1}} f)(x)}{f(x)}\right) = O\left(\frac{T(r, f)}{\log \log r}\right),
\end{eqnarray}
where $r$ tends to infinity outside of a set $F$ satisfying (\ref{eq3.3}).

 Let
    \begin{eqnarray*}
        c(q)=\min\limits_{-n+1\leq k_{2}\leq n-1}\left\{|q^{\frac{k_{2}+1}{2}}-q^{\frac{k_{2}-1}{2}}|, |q^{-\frac{k_{2}+1}{2}}-q^{-\frac{k_{2}-1}{2}}|\right\}>0.
    \end{eqnarray*}
    By letting $|x|$ and hence $|z|$ to be sufficiently large, for any $-n+1\leq k_{3}\leq n-1$, we have
    \begin{eqnarray*}
        \left|\frac{1}{\eta_{q}^{k_{3}+1}x-\eta_{q}^{k_{3}-1}x}\right|&=&\left|\frac{2}{(q^{\frac{k_{3}+1}{2}}-q^{\frac{k_{3}-1}{2}})z-(q^{-\frac{k_{3}+1}{2}}-q^{-\frac{k_{3}-1}{2}})z^{-1}}\right|\\
        &\leq&\frac{2}{c(q)(|z|+|z^{-1}|)}\leq\frac{4}{c(q)|z|}\leq 1,
    \end{eqnarray*}
    which yields
    \begin{eqnarray*}
        m\left(r, \frac{1}{\eta_{q}^{k_{3}+1}x-\eta_{q}^{k_{3}-1}x}\right)=0
    \end{eqnarray*}
    for all sufficiently large $|x|$.
    
    Since $(\mathcal{D}_{q}^{n}f)(x)$ is a combination of $(\eta_{q}^{k_{1}}f)(x)$ and $1/(\eta_{q}^{k_{3}+1}x-\eta_{q}^{k_{3}-1}x)$ when $|x|$ is sufficiently large, then
    \begin{eqnarray}\label{aeq3.2.7.1}
        m\left(r, \frac{(\mathcal{D}_{q}^{n}f)(x)}{f(x)}\right)=O\left(\sum_{-n\leq k_{1}\leq n}m\left(r, \frac{(\eta_{q}^{k_{1}}f)(x)}{f(x)}\right)\right)=O\left(\frac{T(r, f)}{\log\log r}\right),
    \end{eqnarray}
    where $r$ runs to infinity outside of a set $F$ satisfying (\ref{eq3.3}).
\end{proof}

Following the definition of Chiang and Feng \cite{NADV}, assume $a \in \hat{\mathbb{C}}$, $f$ is a meromorphic function on $\mathbb{C}$, and $g: \mathbb{C} \rightarrow \mathbb{C}$ is a map, \textit{not necessarily} entire. The counting function $n(r, 1/(f(g(x)) - a))$ is defined to be the number of $a$-points of $f$, counting multiplicities at the point $g(x)$, in the set $\{g(x): |x| < r\}$. The corresponding integrated counting function $N(r, 1/(f(g(x)) - a))$ can also be denoted in the usual way according to $n(r, 1/(f(g(x)) - a))$. 

Recall that there is an $O(\log r)$ term in the error in \cite[Theorem 3.3]{NADV} because the one-sided Askey-Wilson shift operators $\eta_{q} x$ and $\eta_{q}^{-1} x$ are not invertible for all $x \in \mathbb{C}$. In what follows, we will not delve into this phenomenon but will attribute the error to $O(\log r)$. The following lemma is another version of \cite[Theorem 5.1]{NADV}, which will be frequently used in this paper.

\smallskip

\begin{lemma}\label{le3.3}
    Let $f$ be a non-constant meromorphic function with $\sigma_{2}^{\log}(f)<1$. Then, for each $a\in\hat{\mathbb{C}}$, and $n\in\mathbb{N}$, we have
    \begin{eqnarray}\label{eq3.3.3}
        N\left(r, \frac{1}{(\eta_{q}^{\pm n}f)-a}\right)=N\left(r, \frac{1}{f-a}\right)+S_{\log}(r, f)
    \end{eqnarray}
    where $S_{\log}(r, f)=O\left(\frac{T(r, f)}{\log\log r}\right)+O(\log r)$ and $r$ runs to infinity outside of a set $F$ satisfying (\ref{eq3.3}).
\end{lemma}

\smallskip

\begin{proof}
    We only prove the case that $n=1$, the other cases can be proved similarly. Recall that $\eta_{q}^{-1}\eta_{q}x=\eta_{q}\eta_{q}^{-1}x=x$ for sufficiently large $|x|=r$ and hence $|z|$ is large enough. Since
    \begin{eqnarray*}
        |\eta_{q}^{\pm 1}x|=\left|q^{\pm 1/2}z\right|\left|\frac{1+q^{\mp 1}z^{-2}}{2}\right|\leq 2|q^{\pm 1/2}|r
    \end{eqnarray*}
    when $|x|=r$ is large enough, then by applying Lemma \ref{le3.1} to $N(r, 1/(f-a))$ and $N(r, 1/(\eta_{q}f-a))$ separately, we obtain
    \begin{eqnarray}\label{aeq3.3.1}
        N\left(r, \frac{1}{\eta_{q}f-a}\right)&\leq& N\left(2|q^{1/2}|r, \frac{1}{f-a}\right)+O(\log r)\nonumber\\
        &=&N\left(r, \frac{1}{f-a}\right)+S_{\log}(r, f)
    \end{eqnarray}
    and
    \begin{eqnarray}\label{aeq3.3.2}
        N\left(r, \frac{1}{f-a}\right)&\leq& N\left(r, \frac{1}{\eta_{q}^{-1}(\eta_{q}f)-a}\right)+O(\log r)\nonumber\\
        &\leq&N\left(2|q^{-1/2}|r, \frac{1}{\eta_{q}f-a}\right)+S_{\log}(r, f)\nonumber\\
        &=&N\left(r, \frac{1}{\eta_{q}f-a}\right)+S_{\log}(r, f),
    \end{eqnarray}
    where $r$ runs to infinity outside of a set $F$ satisfying (\ref{eq3.3}). Hence the assertion follows by combining (\ref{aeq3.3.1}) and (\ref{aeq3.3.2}).
\end{proof}

We will adopt the notation $S_{\log}(r, f)$ as in Lemma \ref{le3.3} throughout the remainder of the paper. It is worth noting that "$O$" depends on $n$ and $q$ in Lemma~\ref{le3.3}. Indeed, the dependence of the first "$O$" on $n$ and $q$ can be seen from (\ref{aeq3.2.6}) and (\ref{aeq3.2.7.1}). For the second "$O$", since we always require $|x| > \mathcal{R}(n, q)$ such that the operators $\eta_{q}^{\pm n}$ are analytic and invertible, any other points $x$ satisfying $|x| \leq \mathcal{R}(n, q)$ will be included in the second term $O(\log r)$.

The following theorem, which is also another version of \cite[Theorem 3.3]{NADV}, can be deduced from Lemma \ref{le3.3}.

\smallskip

\begin{theorem}\label{the3.4}
    Let $n\in\mathbb{N}$ and $f$ be a non-constant meromorphic function with $\sigma_{2}^{\log}(f)<1$. Then
    \begin{eqnarray}
        T(r, \mathcal{D}_{q}^{n}f)&\leq& (n+1)T(r, f)+S_{\log}(r, f),\label{eq3.4.1}\\
        T(r, \mathcal{A}_{q^{n}}f)&\leq& 2T(r, f)+S_{\log}(r, f),\label{eq3.4.1.1}\\
        T(r, (\eta_{q^{n}}f)(\eta_{q^{-n}}f))&\leq& 2T(r, f)+S_{\log}(r, f),\label{eq3.4.1.2}
    \end{eqnarray}
    where $r$ runs to infinity outside of a set $F$ satisfying (\ref{eq3.3}). In particular, these imply
    \begin{eqnarray}
        \sigma_{2}^{\log}(\mathcal{D}_{q}^{n}f)&\leq& \sigma_{2}^{\log}(f),\label{eq3.4.2}\\
        \sigma_{2}^{\log}(\mathcal{A}_{q^{n}}f)&\leq& \sigma_{2}^{\log}(f),\label{eq3.4.2.1}\\
        \sigma_{2}^{\log}((\eta_{q^{n}}f)(\eta_{q^{-n}}f))&\leq& \sigma_{2}^{\log}(f).\label{eq3.4.2.2}
    \end{eqnarray}
\end{theorem}

\smallskip

\begin{proof}
    Since $(\mathcal{D}_{q}^{n}f)(x)$ is a combination of $(\eta_{q}^{n}f)(x), (\eta_{q}^{n-2}f)(x), \cdots, (\eta_{q}^{-n}f)(x)$ which has $n+1$ elements when $|x|$ is sufficiently large, then by applying Lemma \ref{le3.3}, we have
    \begin{eqnarray*}
        N(r, \mathcal{D}_{q}^{n}f)\leq \sum_{k=0}^{n}N(r, \eta_{q}^{n-2k}f)+O(\log r)=(n+1)N(r, f)+S_{\log}(r, f).
    \end{eqnarray*}
    Therefore, we obtain
    \begin{eqnarray*}
        T(r, \mathcal{D}_{q}^{n}f)&\leq& m\left(r, \frac{\mathcal{D}_{q}^{n}f}{f}\right)+N(r, \mathcal{D}_{q}^{n}f)+m(r, f)+S_{\log}(r, f)\\
        &\leq & m(r, f)+(n+1)N(r, f)+S_{\log}(r, f)\\
        &\leq&(n+1)T(r, f)+S_{\log}(r, f),
    \end{eqnarray*}
    where $r$ runs to infinity outside of a set $F$ satisfying (\ref{eq3.3}). (\ref{eq3.4.1.1}) can be obtained similarly. Moreover,
    \begin{eqnarray*}
        &&T(r, (\eta_{q^{n}}f)\cdot (\eta_{q^{-n}}f))\\
        &\leq& m\left(r, \frac{(\eta_{q^{n}}f)\cdot (\eta_{q^{-n}}f)}{f^{2}}\right)+N(r, (\eta_{q^{n}}f)\cdot (\eta_{q^{-n}}f))+m(r, f^{2})+S_{\log}(r, f)\\
        &\leq&2m(r, f)+2N(r, f)+S_{\log}(r, f)\\
        &\leq&2T(r, f)+S_{\log}(r, f),
    \end{eqnarray*}
    where $r$ runs to infinity outside of a set $F$ satisfying (\ref{eq3.3}). Thus, the assertion follows.
\end{proof}

Since we are going to use the AW operator to act on a holomorphic curve in the projective space $\mathbb{P}^{n}(\mathbb{C})$, we first prove the following theorem to ensure that the Askey-Wilson operator is well-defined in $\mathbb{P}^{n}(\mathbb{C})$.

\smallskip

\begin{theorem}\label{the3.5}
    Let $n\in\mathbb{N}$ and $f:\mathbb{C}\rightarrow\mathbb{P}^{n}(\mathbb{C}) $ be a holomorphic curve satisfying
    \begin{eqnarray}\label{eq3.5.1}
        \limsup\limits_{r\rightarrow\infty}\frac{\log^{+}\log^{+} T_{f}(r)}{\log\log r}<1.
    \end{eqnarray}
    Then, there exists a reduced representation $\textbf{f}=(f_{0}, \cdots, f_{n})$ of $f$ such that $\sigma_{2}^{\log}(f_{i})<1$ for $i=0, \cdots, n$.
\end{theorem}

\smallskip

To prove the theorem, we need following lemma which is from the Remark 1 of \cite[Theorem 3.3]{GO08}.

\smallskip

\begin{lemma}\label{le3.6}
    If $h(z)$ is a Weierstrass canonical product of genus $p$, then
    \begin{eqnarray}\label{eq3.6}
        \log M(r, h)\leq C(p)\left\{r^{p}\int_{0}^{r}\frac{n(t, 1/h)}{t^{p+1}}dt+r^{p+1}\int_{r}^{\infty}\frac{n(t, 1/h)}{t^{p+2}}dt\right\}
    \end{eqnarray}
    where $M(r, h)=\max_{|z|=r}\{|h(z)|\}$. In particular, $C(0)=1$.
\end{lemma}

\smallskip

\begin{proof}[Proof of Theorem \ref{the3.5}]
    Suppose $\textbf{g}=(g_{0}, \cdots, g_{n})$ is a reduced representation of $f$. If some of $g_{i}$, $i=0, \cdots, n$ have finitely many zeros, without loss of generality, we assume $g_{0}$ is this kind of $g_{i}$. If $g_{0}$ does not have any zeros, we define an auxiliary function $h_{1}(z)=1$, otherwise we assume $\{z_{k}\}_{k=1}^{\mu}$ are all $g_{0}$'s zeros, $\mu\in\mathbb{N}$, and let
    \begin{eqnarray*}
        h_{1}(z):=\prod_{k=1}^{\mu}\left(1-\frac{z}{z_{k}}\right).
    \end{eqnarray*}
    Then we have $\sigma_{2}^{\log}(h_{1})<1$. Since $h_{1}(z)$ and $g_{0}$ have the same zeros, then $h_{1}/g_{0}$ is a non-zero entire function, and
    \begin{eqnarray*}
        \textbf{f}=(f_{0}, \cdots, f_{n}):=\frac{h_{1}}{g_{0}}(g_{0}, \cdots, g_{n})=(h_{1}, \frac{g_{1}}{g_{0}}h_{1}, \cdots, \frac{g_{n}}{g_{0}}h_{1})
    \end{eqnarray*}
    is also a reduced representation of $f$.
    Then it follows by applying (\ref{equa2.3}) that
\begin{eqnarray}\label{aeq3.6.2}
    T(r, f_{i}) = T\left(r, \frac{g_{i}}{g_{0}} h_{1}\right) &\leq& T\left(r, \frac{g_{i}}{g_{0}}\right) + T(r, h_{1}) + O(1)\nonumber\\
    &\leq& T_{f}(r) + T(r, h_{1}) + O(1)
\end{eqnarray}
for $i = 0, \cdots, n$. Thus, we can  conclude that $\sigma_{2}^{\log }(f_{i}) < 1$ for all $i = 0, \cdots, n$ by combining (\ref{eq3.5.1}), (\ref{aeq3.6.2}), and the fact that $\sigma_{2}^{\log}(h_{1}) < 1$, which proves the assertion.

    Now we assume all $g_{i}$ have infinitely many zeros. We can find at least one $g_{i}$ such that $z=0$ is not the zero of $g_{i}$. Without loss of generality, we assume $g_{0}$ is this kind of $g_{i}$ and $\{z_{l}\}_{l=1}^{\infty}$ are all $g_{0}$'s zeros, counting multiplicities, and $0\not\in \{z_{l}\}_{l=1}^{\infty}$. Combining (\ref{equa2.2}) and (\ref{eq3.5.1}), we have that the exponent of convergence of $g_{0}$ is less than $1$, which implies
    \begin{eqnarray*}
        \sum_{l=1}^{\infty}\left|\frac{1}{z_{l}}\right|<\infty.
    \end{eqnarray*}
    Hence
    \begin{eqnarray*}
        h_{2}(z):=\prod_{l=1}^{\infty}\left(1-\frac{z}{z_{l}}\right)
    \end{eqnarray*}
    converges absolutely and uniformly on each bounded disc, thus $h_{2}(z)$ is a Weierstrass canonical product of genus $0$ and $h_{2}(z)$ has zeros precisely at the points $\{z_{l}\}_{l=1}^{\infty}$, that is, $n(r, 1/h_{2})\equiv n(r, 1/g_{0})$ for all $r>0$. Then it follows from Lemma \ref{le3.6} that
    \begin{eqnarray*}
        T(r, h)\leq \log M(r, h)&\leq& \int_{0}^{r}\frac{n(t, 1/h)}{t}dt+r\int_{r}^{\infty}\frac{n(t, 1/h)}{t^{2}}dt\\
        &=&N\left(r, \frac{1}{g_{0}}\right)+r\int_{r}^{\infty}\frac{n(t, 1/g_{0})}{t^{2}}dt.
    \end{eqnarray*}
    Now we estimate the second term. Let $R=r$ in (\ref{aeq3.2.3}) and apply (\ref{equa2.2}) again; we have
    \begin{eqnarray*}
        \limsup\limits_{r\rightarrow \infty}\frac{\log\log n(r, 1/g_{0})}{\log\log r}&\leq& \limsup\limits_{r\rightarrow \infty}\frac{\log\log\frac{N(r\log r, 1/g_{0})}{\log\log r}}{\log\log r}\\
        &=&\limsup\limits_{r\rightarrow \infty}\frac{\log\log N(r, 1/g_{0})}{\log\log r}:=\eta<1.
    \end{eqnarray*}
    Let $\kappa_{1}\in(\eta, 1)$ and $\kappa\in(\kappa_{1}, 1)$, then $\frac{\log\log n(r, 1/g_{0})}{\log\log r}<\kappa_{1}$ when $r$ is sufficiently large, and thus $n(r, 1/g_{0})<e^{(\log r)^{\kappa_{1}}}$. Then
    \begin{eqnarray*}
        r\int_{r}^{\infty}\frac{n(t, 1/g_{0})}{t^{2}}dt &\leq &r\int_{r}^{\infty}\frac{e^{(\log r)^{\kappa_{1}}}}{t^{2}}dt\\
        &\leq& 2r\int_{r}^{\infty}\frac{e^{(\log t)^{\kappa}}\left(1-\frac{\kappa}{(\log t)^{1-\kappa}}\right)}{t^{2}}dt\\
        &=&2r\left(-\frac{e^{(\log t)^{\kappa}}}{t}\right)\bigg|_{t=r}^{\infty}=2e^{(\log r)^{\kappa}}
    \end{eqnarray*}
    for all sufficiently large $r$. Therefore,
    \begin{eqnarray*}
        &&\limsup\limits_{r\rightarrow \infty}\frac{\log\log T(r, h)}{\log\log r}\nonumber\\
        &\leq& \max\left\{\limsup\limits_{r\rightarrow \infty}\frac{\log\log N(r, 1/g_{0})}{\log\log r}, \limsup\limits_{r\rightarrow \infty}\frac{\log\log (2e^{(\log r)^{\kappa}})}{\log\log r}\right\}\nonumber\\
        &=&\kappa<1.
    \end{eqnarray*}
    Now the assertion follows by replacing $h_{1}(z)$ by $h_{2}(z)$ in the above analysis.
\end{proof}

\section{Askey-Wilson version of Wronskian-Casorati determinant}\label{sec4}

The Wronskian determinant plays an important role in proving the Second Main Theorem for holomorphic curves in projective space. Hence, we first introduce following \textit{Askey-Wilson version of the Wronskian-Casorati determinant} with respect to meromorphic functions $f_{0}, \cdots, f_{n}$. Let
\begin{eqnarray}\label{eq4.1.2}
    \mathcal{W}(f_{0}, \cdots, f_{n})(x):=\begin{vmatrix}
    \mathcal{A}_{q^{n}}f_{0} & \cdots & \mathcal{A}_{q^{n}}f_{n} \\ 
    \mathcal{A}_{q^{n-1}}\mathcal{D}_{q}f_{0} & \cdots & \mathcal{A}_{q^{n-1}}\mathcal{D}_{q}f_{n} \\
    \vdots & & \vdots\\
    \mathcal{D}_{q}^{n}f_{0} & \cdots & \mathcal{D}_{q}^{n}f_{n}
    \end{vmatrix}(x).
\end{eqnarray}
According to the definition, we immediately have following theorem by applying Theorem \ref{pro2.3}.

\smallskip

\begin{theorem}\label{the4.1.0}
    Let $f_{0}, \cdots, f_{n}$ be meromorphic (entire) functions. Then $\mathcal{W}(f_{0}, \cdots, f_{n})$ is meromorphic (entire).
\end{theorem}

\smallskip

In addition, the Askey-Wilson version of the Wronskian-Casorati determinant has following properties which also can be regarded as another equivalent definition when $|x|$ is sufficiently large.

\smallskip

\begin{proposition}\label{pro4.1}
    Suppose $f_{0}, \cdots, f_{n}$ are meromorphic functions, then
    \begin{eqnarray}
    &&\mathcal{W}(f_{0}, \cdots, f_{n})(x)\nonumber\\
    &=&\begin{vmatrix}
    \eta_{q}^{\delta_{0}n}f_{0} & \cdots & \eta_{q}^{\delta_{0}n}f_{n} \\ 
    \eta_{q}^{\delta_{1}(n-1)}\mathcal{D}_{q}f_{0} & \cdots & \eta_{q}^{\delta_{1}(n-1)}\mathcal{D}_{q}f_{n} \\
    \vdots & & \vdots\\
    \eta_{q}^{\delta_{n-1}}\mathcal{D}_{q}^{n-1}f_{0} & \cdots & \eta_{q}^{\delta_{n-1}}\mathcal{D}_{q}^{n-1}f_{n} \\
    \mathcal{D}_{q}^{n}f_{0} & \cdots & \mathcal{D}_{q}^{n}f_{n}
    \end{vmatrix}(x)\label{eq4.1.1}\\
    &=&\begin{vmatrix}
    \eta_{q}^{n}f_{0} & \cdots & \eta_{q}^{n}f_{n} \\ 
    \eta_{q}^{n-2}f_{0} & \cdots & \eta_{q}^{n-2}f_{n} \\
    \vdots & & \vdots\\
    \eta_{q}^{-n}f_{0} & \cdots & \eta_{q}^{-n}f_{n}
    \end{vmatrix}(x)\cdot\prod_{i=1}^{n}\prod_{j=1}^{i}\frac{-1}{\eta_{q}^{i-2j+2}x-\eta_{q}^{i-2j}x}\label{eq4.1.3}
    \end{eqnarray}
    for all $x\in\mathbb{C}$ satisfying $|x|>\mathcal{R}(n, q)$ where $\{\delta_{i}\}_{i=0}^{n-1}\subset\{\pm 1\}$.
\end{proposition}

\smallskip

\begin{proof}
    Since $|x|>\mathcal{R}(n, q)$, then $\eta_{q}^{\pm i}x$ are analytic and invertible for all $i=1, \cdots, n$ and $\eta_{q}^{j+1}x\neq \eta_{q}^{j-1}x$ for $j=-n+1, -n+2, \cdots, n-1$. We first claim that for each $f\in\{\mathcal{D}_{q}^{i}f_{j}\}_{i, j=0}^{n}$ and $\nu=1, \cdots, n-i-1$, there exist some functions $A_{k, \nu}(x)$ only depending on $x$, $k=1, \cdots, \nu$ such that
    \begin{eqnarray}\label{aeq4.1.1}
        (\eta_{q}^{\nu}f)(x)-(\eta_{q}^{-\nu}f)(x)=\sum_{k=1}^{\nu}A_{k, \nu}(x)(\eta_{q}^{\nu-k}\mathcal{D}_{q}^{k}f)(x).
    \end{eqnarray}
    For instance, $A_{1, 1}(x)=\eta_{q}x-\eta_{q}^{-1}x$. To prove the claim (\ref{aeq4.1.1}), we define a matrix omitting $x$
    \begin{eqnarray*}
        \left(a_{i, j}^{(\nu)}\right)_{\nu\times \nu}=\left(\begin{matrix}
    \eta_{q}^{\nu-1}\mathcal{D}_{q}f &   &  &  \\ 
    \eta_{q}^{\nu-2}\mathcal{D}_{q}^{2}f & \eta_{q}^{\nu-3}\mathcal{D}_{q}f & & \\
    \vdots & \vdots & \ddots\\
    \mathcal{D}_{q}^{\nu}f & \eta_{q}^{-1}\mathcal{D}_{q}^{\nu-1}f & \cdots & \eta_{q}^{-\nu+1}\mathcal{D}_{q}f
    \end{matrix}\right),
    \end{eqnarray*}
    from which we can see that $a_{i, j}^{(\nu)}=\eta_{q}^{\nu-i-j+1}\mathcal{D}_{q}^{i-j+1}f$ when $1\leq j\leq i\leq\nu$. Note that
    \begin{eqnarray*}
        &&a_{i+1, j+1}^{(\nu)}\\
        &=&\eta_{q}^{\nu-i-j+1}\mathcal{D}_{q}^{i-j+1}f-\left(\eta_{q}^{\nu-i-j+1}\mathcal{D}_{q}^{i-j+1}f-\eta_{q}^{\nu-i-j-1}\mathcal{D}_{q}^{i-j+1}f\right)\\
        &=&a_{i, j}^{(\nu)}-\eta_{q}^{\nu-i-j}\left(\frac{\eta_{q}\mathcal{D}_{q}^{i-j+1}f-\eta_{q}^{-1}\mathcal{D}_{q}^{i-j+1}f}{\eta_{q}x-\eta_{q}^{-1}x}\right)\left(\eta_{q}^{\nu-i-j+1}x-\eta_{q}^{\nu-i-j-1}x\right)\\
        &=&a_{i, j}^{(\nu)}-a_{i+1, j}^{(\nu)}\left(\eta_{q}^{\nu-i-j+1}x-\eta_{q}^{\nu-i-j-1}x\right)
    \end{eqnarray*}
    for all $1\leq j\leq i\leq \nu-1$. This implies there is a relationship for each triangular region of $(a_{i, j}^{(\nu)})_{\nu\times \nu}$. Hence, by iteration, we can see that all elements on the diagonal can be represented by elements in the first column with some coefficients only about $x$, which means there exist some functions $B_{k, i, \nu}(x)$,$1\leq k\leq \min\{i, \nu\}$ such that $a_{i, i}^{(\nu)}=\sum_{k=1}^{i}B_{k, i, \nu}(x)a_{k,1}^{(\nu)}$, that is,
    \begin{eqnarray}\label{aeq4.1.2}
        \eta_{q}^{\nu-2i+1}\mathcal{D}_{q}f=\sum_{k=1}^{i}B_{k, i, \nu}(x)\eta_{q}^{\nu-k}\mathcal{D}_{q}^{k}f,
    \end{eqnarray}
    $i=1, \cdots, \nu$. Furthermore, we have
    \begin{eqnarray}\label{aeq4.1.3}
        \eta_{q}^{\nu}f-\eta_{q}^{-\nu}f&=&\sum_{i=1}^{\nu}(\eta_{q}^{\nu-2i+2}f-\eta_{q}^{\nu-2i}f)\nonumber\\
        &=&\sum_{i=1}^{\nu}\eta_{q}^{\nu-2i+1}\left(\frac{\eta_{q}f-\eta_{q}^{-1}f}{\eta_{q}x-\eta_{q}^{-1}x}\right)\cdot(\eta_{q}^{\nu-2i+2}x-\eta_{q}^{\nu-2i}x)\nonumber\\
        &=&\sum_{i=1}^{\nu}\eta_{q}^{\nu-2i+1}\mathcal{D}_{q}f\cdot(\eta_{q}^{\nu-2i+2}x-\eta_{q}^{\nu-2i}x).
    \end{eqnarray}
    Then the claim (\ref{aeq4.1.1}) follows by combining (\ref{aeq4.1.2}) and (\ref{aeq4.1.3}).

    Now we prove
    \begin{eqnarray}\label{aeq4.1.0}
        \begin{vmatrix}
          \eta_{q}^{n}f_{0} & \cdots & \eta_{q}^{n}f_{n} \\ 
          \eta_{q}^{n-1}\mathcal{D}_{q}f_{0} & \cdots & \eta_{q}^{n-1}\mathcal{D}_{q}f_{n} \\
          \vdots & \ddots &\vdots\\
          \eta_{q}\mathcal{D}_{q}^{n-1}f_{0} &  \cdots & \eta_{q}\mathcal{D}_{q}^{n-1}f_{n} \\
          \mathcal{D}_{q}^{n}f_{0} & \cdots & \mathcal{D}_{q}^{n}f_{n}
       \end{vmatrix}=\begin{vmatrix}
    \eta_{q}^{\delta_{0}n}f_{0} & \cdots & \eta_{q}^{\delta_{0}n}f_{n} \\ 
    \eta_{q}^{\delta_{1}(n-1)}\mathcal{D}_{q}f_{0} & \cdots & \eta_{q}^{\delta_{1}(n-1)}\mathcal{D}_{q}f_{n} \\
    \vdots & & \vdots\\
    \eta_{q}^{\delta_{n-1}}\mathcal{D}_{q}^{n-1}f_{0} & \cdots & \eta_{q}^{\delta_{n-1}}\mathcal{D}_{q}^{n-1}f_{n} \\
    \mathcal{D}_{q}^{n}f_{0} & \cdots & \mathcal{D}_{q}^{n}f_{n}
    \end{vmatrix}
    \end{eqnarray}
    for all $x\in\mathbb{C}$ satisfying $x>\mathcal{R}(n, q)$ and $\{\delta_{i}\}_{i=0}^{n-1}\subset\{\pm 1\}$. If $\delta_{0}=\cdots=\delta_{n-1}=1$, then there is nothing to prove. Thus we can assume there are some $\delta_{i}=-1$, and suppose $i^{*}\in\{0, \cdots, n-1\}$ is the largest number such that $\delta_{i^{*}}=-1$. Let the $(i^{*}+1)$-th row subtract the combination from $(i^{*}+2)$-th row to $(n+1)$-th row according to the claim (\ref{aeq4.1.1}), then we have
    \begin{eqnarray*}
        \begin{vmatrix}
    \eta_{q}^{\delta_{0}n}f_{0} & \cdots & \eta_{q}^{\delta_{0}n}f_{n} \\ 
    \vdots & & \vdots\\
    \eta_{q}^{-(n-i^{*})}\mathcal{D}_{q}^{i^{*}}f_{0} & \cdots & \eta_{q}^{-(n-i^{*})}\mathcal{D}_{q}^{i^{*}}f_{n} \\
    \vdots & & \vdots\\
    \mathcal{D}_{q}^{n}f_{0} & \cdots & \mathcal{D}_{q}^{n}f_{n}
    \end{vmatrix}=\begin{vmatrix}
    \eta_{q}^{\delta_{0}n}f_{0} & \cdots & \eta_{q}^{\delta_{0}n}f_{n} \\ 
    \vdots & & \vdots\\
    \eta_{q}^{n-i^{*}}\mathcal{D}_{q}^{i^{*}}f_{0} & \cdots & \eta_{q}^{n-i^{*}}\mathcal{D}_{q}^{i^{*}}f_{n} \\
    \vdots & & \vdots\\
    \mathcal{D}_{q}^{n}f_{0} & \cdots & \mathcal{D}_{q}^{n}f_{n}
    \end{vmatrix}.
    \end{eqnarray*}
    The procedure can be regarded in the way that we change $\delta_{i^{*}}$ from $-1$ to $1$ but will not change the value of the determinant. Since we only need the last several rows in the calculation, we can deal with the adjacent row that $\delta_{i}=-1$ with the same analysis until all $\delta_{i}=1$, $i=0, \cdots, n-1$. Hence, (\ref{aeq4.1.0}) follows.

    Now we finally consider (\ref{eq4.1.1}).
    \begin{eqnarray*}
        &&\begin{vmatrix}
    \mathcal{A}_{q^{n}}f_{0} & \cdots & \mathcal{A}_{q^{n}}f_{n} \\ 
    \vdots & & \vdots\\
    \mathcal{A}_{q}\mathcal{D}_{q}^{n-1}f_{0} & \cdots & \mathcal{A}_{q}\mathcal{D}_{q}^{n-1}f_{n} \\
    \mathcal{D}_{q}^{n}f_{0} & \cdots & \mathcal{D}_{q}^{n}f_{n}
    \end{vmatrix}\\
    &=&\frac{1}{2}\begin{vmatrix}
    \mathcal{A}_{q^{n}}f_{0} & \cdots & \mathcal{A}_{q^{n}}f_{n} \\ 
    \vdots & & \vdots\\
    \eta_{q}\mathcal{D}_{q}^{n-1}f_{0}+ \eta_{q}^{-1}\mathcal{D}_{q}^{n-1}f_{0}& \cdots & \eta_{q}\mathcal{D}_{q}^{n-1}f_{n}+\eta_{q}^{-1}\mathcal{D}_{q}^{n-1}f_{n} \\
    \mathcal{D}_{q}^{n}f_{0} & \cdots & \mathcal{D}_{q}^{n}f_{n}
    \end{vmatrix}\\
    &=&\frac{1}{2}\begin{vmatrix}
    \mathcal{A}_{q^{n}}f_{0} & \cdots & \mathcal{A}_{q^{n}}f_{n} \\ 
    \vdots & & \vdots\\
    \mathcal{A}_{q^{2}}\mathcal{D}_{q}^{n-2}f_{0} & \dots &\mathcal{A}_{q^{2}}\mathcal{D}_{q}^{n-2}f_{n}\\
    \eta_{q}\mathcal{D}_{q}^{n-1}f_{0}& \cdots & \eta_{q}\mathcal{D}_{q}^{n-1}f_{n} \\
    \mathcal{D}_{q}^{n}f_{0} & \cdots & \mathcal{D}_{q}^{n}f_{n}
    \end{vmatrix}+\frac{1}{2}\begin{vmatrix}
    \mathcal{A}_{q^{n}}f_{0} & \cdots & \mathcal{A}_{q^{n}}f_{n} \\ 
    \vdots & & \vdots\\
    \mathcal{A}_{q^{2}}\mathcal{D}_{q}^{n-2}f_{0} & \dots &\mathcal{A}_{q^{2}}\mathcal{D}_{q}^{n-2}f_{n}\\
    \eta_{q}^{-1}\mathcal{D}_{q}^{n-1}f_{0}& \cdots &\eta_{q}^{-1}\mathcal{D}_{q}^{n-1}f_{n} \\
    \mathcal{D}_{q}^{n}f_{0} & \cdots & \mathcal{D}_{q}^{n}f_{n}
    \end{vmatrix}\\
    &=&\cdots \\
    &=&\frac{1}{2^{n}}\sum_{\{\delta_{i}\}_{i=0}^{n-1}\subset\{\pm 1\}}\begin{vmatrix}
    \eta_{q}^{\delta_{0}n}f_{0} & \cdots & \eta_{q}^{\delta_{0}n}f_{n} \\ 
    \eta_{q}^{\delta_{1}(n-1)}\mathcal{D}_{q}f_{0} & \cdots & \eta_{q}^{\delta_{1}(n-1)}\mathcal{D}_{q}f_{n} \\
    \vdots & & \vdots\\
    \eta_{q}^{\delta_{n-1}}\mathcal{D}_{q}^{n-1}f_{0} & \cdots & \eta_{q}^{\delta_{n-1}}\mathcal{D}_{q}^{n-1}f_{n} \\
    \mathcal{D}_{q}^{n}f_{0} & \cdots & \mathcal{D}_{q}^{n}f_{n}
    \end{vmatrix},
    \end{eqnarray*}
    which deduces (\ref{eq4.1.1}) by following (\ref{aeq4.1.0}).
    
    For the second equality (\ref{eq4.1.3}), we first define
    \begin{eqnarray*}
        H_{k}(x)=:\prod_{i=1}^{n-k}\prod_{j=1}^{i}\frac{-1}{\eta_{q}^{i-2j+2}x-\eta_{q}^{i-2j}x}
    \end{eqnarray*}
    for $k=0, \cdots, n-1$. Now we proceed to prove
    \begin{eqnarray*}
        \mathcal{W}(f_{0}, \cdots, f_{n})=\begin{vmatrix}
    \eta_{q}^{n}f_{0} & \cdots & \eta_{q}^{n}f_{n} \\ 
    \eta_{q}^{n-1}\mathcal{D}_{q}f_{0} & \cdots & \eta_{q}^{n-1}\mathcal{D}_{q}f_{n} \\
    \vdots & & \vdots\\
    \eta_{q}^{n-k}\mathcal{D}_{q}^{k}f_{0} & \cdots &\eta_{q}^{n-k}\mathcal{D}_{q}^{k}f_{n}\\
    \eta_{q}^{n-k-2}\mathcal{D}_{q}^{k}f_{0} & \cdots &\eta_{q}^{n-k-2}\mathcal{D}_{q}^{k}f_{n}\\
    \vdots & & \vdots\\
    \eta_{q}^{-n+k}\mathcal{D}_{q}^{k}f_{0} & \cdots &\eta_{q}^{-n+k}\mathcal{D}_{q}^{k}f_{n}
    \end{vmatrix}\cdot H_{k}(x)
    \end{eqnarray*}
    for $k=n-1, \cdots, 0$ when $|x|>\mathcal{R}(n, q)$ by reverse induction. Then (\ref{eq4.1.3}) follows when $k=0$. If $k=n-1$, then from (\ref{eq4.1.1}) when $\delta_{i}=1$, $i=0, \cdots, n-1$, we have
    \begin{eqnarray*}
        \mathcal{W}(f_{0}, \dots, f_{n})&=&\begin{vmatrix}
            \eta_{q}^{n}f_{0} & \cdots &\eta_{q}^{n}f_{n}\\
            \vdots& &\vdots\\
            \eta_{q}\mathcal{D}_{q}^{n-1}f_{0}&\cdots &\eta_{q}\mathcal{D}_{q}^{n-1}f_{n}\\
            \frac{\eta_{q}\mathcal{D}_{q}^{n-1}f_{0}-\eta_{q}^{-1}\mathcal{D}_{q}^{n-1}f_{0}}{\eta_{q}x-\eta_{q}^{-1}x}& \cdots &\frac{\eta_{q}\mathcal{D}_{q}^{n-1}f_{n}-\eta_{q}^{-1}\mathcal{D}_{q}^{n-1}f_{n}}{\eta_{q}x-\eta_{q}^{-1}x}
        \end{vmatrix}\\
        &=&\begin{vmatrix}
            \eta_{q}^{n}f_{0} & \cdots &\eta_{q}^{n}f_{n}\\
            \vdots& &\vdots\\
            \eta_{q}\mathcal{D}_{q}^{n-1}f_{0}&\cdots &\eta_{q}\mathcal{D}_{q}^{n-1}f_{n}\\
            \eta_{q}^{-1}\mathcal{D}_{q}^{n-1}f_{0}& \cdots &\eta_{q}^{-1}\mathcal{D}_{q}^{n-1}f_{n}
        \end{vmatrix}\cdot\frac{-1}{\eta_{q}x-\eta_{q}^{-1}x},
    \end{eqnarray*}
    which proves the case that $k=n-1$. Now we assume the equality holds for $k+1$, then
    \begin{eqnarray*}
        &&\mathcal{W}(f_{0}, \cdots, f_{n})=\begin{vmatrix}
    \eta_{q}^{n}f_{0} & \cdots & \eta_{q}^{n}f_{n} \\ 
    \eta_{q}^{n-1}\mathcal{D}_{q}f_{0} & \cdots & \eta_{q}^{n-1}\mathcal{D}_{q}f_{n} \\
    \vdots & & \vdots\\
    \eta_{q}^{n-k}\mathcal{D}_{q}^{k}f_{0} & \cdots &\eta_{q}^{n-k}\mathcal{D}_{q}^{k}f_{n}\\
    \eta_{q}^{n-k-1}\mathcal{D}_{q}^{k+1}f_{0} & \cdots &\eta_{q}^{n-k-1}\mathcal{D}_{q}^{k+1}f_{n}\\
    \eta_{q}^{n-k-3}\mathcal{D}_{q}^{k+1}f_{0} & \cdots &\eta_{q}^{n-k-3}\mathcal{D}_{q}^{k+1}f_{n}\\
    \vdots & & \vdots\\
    \eta_{q}^{-n+k+1}\mathcal{D}_{q}^{k+1}f_{0} & \cdots &\eta_{q}^{-n+k+1}\mathcal{D}_{q}^{k+1}f_{n}
    \end{vmatrix}\cdot H_{k+1}(x)
    \end{eqnarray*}
    \begin{eqnarray*}
    &=&\begin{vmatrix}
    \eta_{q}^{n}f_{0} & \cdots & \eta_{q}^{n}f_{n} \\ 
    \eta_{q}^{n-1}\mathcal{D}_{q}f_{0} & \cdots & \eta_{q}^{n-1}\mathcal{D}_{q}f_{n} \\
    \vdots & & \vdots\\
    \eta_{q}^{n-k}\mathcal{D}_{q}^{k}f_{0} & \cdots &\eta_{q}^{n-k}\mathcal{D}_{q}^{k}f_{n}\\
    \frac{\eta_{q}^{n-k}\mathcal{D}_{q}^{k}f_{0}-\eta_{q}^{n-k-2}\mathcal{D}_{q}^{k}f_{0}}{\eta_{q}^{n-k}x-\eta_{q}^{n-k-2}x} & \cdots &\frac{\eta_{q}^{n-k}\mathcal{D}_{q}^{k}f_{n}-\eta_{q}^{n-k-2}\mathcal{D}_{q}^{k}f_{n}}{\eta_{q}^{n-k}x-\eta_{q}^{n-k-2}x}\\
    \vdots & & \vdots\\
    \eta_{q}^{-n+k+1}\mathcal{D}_{q}^{k+1}f_{0} & \cdots &\eta_{q}^{-n+k+1}\mathcal{D}_{q}^{k+1}f_{n}
    \end{vmatrix}\cdot H_{k+1}(x)\\
    &=&\begin{vmatrix}
    \eta_{q}^{n}f_{0} & \cdots & \eta_{q}^{n}f_{n} \\ 
    \vdots & & \vdots\\
    \eta_{q}^{n-k}\mathcal{D}_{q}^{k}f_{0} & \cdots &\eta_{q}^{n-k}\mathcal{D}_{q}^{k}f_{n}\\
    \eta_{q}^{n-k-2}\mathcal{D}_{q}^{k}f_{0} & \cdots &\eta_{q}^{n-k-2}\mathcal{D}_{q}^{k}f_{n}\\
    \vdots & & \vdots\\
    \eta_{q}^{-n+k+1}\mathcal{D}_{q}^{k+1}f_{0} & \cdots &\eta_{q}^{-n+k+1}\mathcal{D}_{q}^{k+1}f_{n}
    \end{vmatrix}\cdot H_{k+1}(x) \frac{-1}{\eta_{q}^{n-k}x-\eta_{q}^{n-k-2}x}\\
    &=&\cdots
    \end{eqnarray*}
    \begin{eqnarray*}
    &=&\begin{vmatrix}
    \eta_{q}^{n}f_{0} & \cdots & \eta_{q}^{n}f_{n} \\ 
    \eta_{q}^{n-1}\mathcal{D}_{q}f_{0} & \cdots & \eta_{q}^{n-1}\mathcal{D}_{q}f_{n} \\
    \vdots & & \vdots\\
    \eta_{q}^{n-k}\mathcal{D}_{q}^{k}f_{0} & \cdots &\eta_{q}^{n-k}\mathcal{D}_{q}^{k}f_{n}\\
    \eta_{q}^{n-k-2}\mathcal{D}_{q}^{k}f_{0} & \cdots &\eta_{q}^{n-k-2}\mathcal{D}_{q}^{k}f_{n}\\
    \vdots & & \vdots\\
    \eta_{q}^{-n+k}\mathcal{D}_{q}^{k}f_{0} & \cdots &\eta_{q}^{-n+k}\mathcal{D}_{q}^{k}f_{n}
    \end{vmatrix}\cdot H_{k+1}(x)\prod_{j=1}^{n-k}\frac{-1}{\eta_{q}^{n-k-2j+2}x-\eta_{q}^{n-k-2j}x}.
    \end{eqnarray*}
    Then the induction step follows by
    \begin{eqnarray*}
        H_{k}(x)=H_{k+1}(x)\cdot \prod_{j=1}^{n-k}\frac{-1}{\eta_{q}^{n-k-2j+2}x-\eta_{q}^{n-k-2j}x}.
    \end{eqnarray*}
    This completes the proof.
\end{proof}

Similar to the Wronskian determinant (\cite[Proposition 1.4.3]{wronskian}), we also have the following properties for $\mathcal{W}(f_{0}, \cdots, f_{n})$.

\smallskip

\begin{proposition}\label{pro4.3.1}
    Let $f_{0}, \cdots, f_{n}, g$ be meromorphic functions and $c_{0}, \cdots, c_{n}$ be nonzero complex constants. Then, for all $x\in\mathbb{C}$ satisfying $|x|>\mathcal{R}(n, q)$, we have

    (i) $\mathcal{W}(c_{0}f_{0}, \cdots, c_{n}f_{n})(x)=c_{0}\cdots c_{n}\mathcal{W}(f_{0}, \cdots, f_{n})(x)$,

    (ii) $\mathcal{W}(1, f_{1}, \cdots, f_{n})(x)=\mathcal{W}(\mathcal{D}_{q}f_{1}, \cdots, \mathcal{D}_{q}f_{n})(x)$,

    (iii) $\mathcal{W}(f_{0}g, \cdots, f_{n}g)(x)=\prod_{k=0}^{n}(\eta_{q}^{n-2k}g)(x)\mathcal{W}(f_{0}, \cdots, f_{n})(x)$,

    (iv) $\mathcal{W}(f_{0}, \cdots, f_{n})(x)=\prod_{k=0}^{n}(\eta_{q}^{n-2k}f_{0})(x)\mathcal{W}\left(\mathcal{D}_{q}\frac{f_{1}}{f_{0}}, \cdots, \mathcal{D}_{q}\frac{f_{n}}{f_{0}}\right)(x)$.
\end{proposition}

\smallskip

\begin{proof}
For (i) and (ii), there is nothing to prove as they follow from basic calculation rules for determinants. For (iii), it can also be easily obtained by following the equality (\ref{eq4.1.3}), so we omit the details here. For (iv), it follows from (ii) and (iii) that
\begin{eqnarray*}
    \mathcal{W}(f_{0}, \cdots, f_{n})&=&(\eta_{q}^{n}f_{0})(\eta_{q}^{n-2}f_{0})\cdots (\eta_{q}^{-n}f_{0})\mathcal{W}\left(1, \frac{f_{1}}{f_{0}}, \cdots, \frac{f_{n}}{f_{0}}\right)\\
    &=&(\eta_{q}^{n}f_{0})(\eta_{q}^{n-2}f_{0})\cdots (\eta_{q}^{-n}f_{0})\mathcal{W}\left(\mathcal{D}_{q}\frac{f_{1}}{f_{0}}, \cdots, \mathcal{D}_{q}\frac{f_{n}}{f_{0}}\right).
\end{eqnarray*}
\end{proof}

\smallskip

\begin{lemma}\label{le4.4.1}

Let $f_{0}, \cdots, f_{n}$ be linearly independent entire functions. Then we have $\mathcal{W}(f_{0}, \cdots, f_{n})(x) \not\equiv 0$.
\end{lemma}

\smallskip

\begin{proof}
    Assume the assertion is not valid, then we have $\mathcal{W}(f_{0}, \cdots, f_{n})(x)\equiv 0$ for all $x\in\mathbb{C}$ satisfying $|x|>\mathcal{R}(n, q)$.  Hence, from (\ref{eq4.1.3}), we know that
    \begin{eqnarray*}
        \left(\begin{matrix}
    \eta_{q}^{n}f_{0} & \cdots & \eta_{q}^{n}f_{n} \\ 
    \eta_{q}^{n-2}f_{0} & \cdots & \eta_{q}^{n-2}f_{n} \\
    \vdots & & \vdots\\
    \eta_{q}^{-n}f_{0} & \cdots & \eta_{q}^{-n}f_{n}
    \end{matrix}\right) \left(\begin{matrix}
    a_{0}\\ 
    a_{1}\\
    \vdots\\
    a_{n}
    \end{matrix}\right)=\left(\begin{matrix}
    0\\ 
    0\\
    \vdots\\
    0
    \end{matrix}\right)
    \end{eqnarray*}
    has a non-trivial solution $a_{0}, \cdots, a_{n}\in\mathbb{C}$ for all $x\geq \mathcal{R}(n, q)$.
    
    Let $g(x) = a_{0} f_{0}(x) + \cdots + a_{n} f_{n}(x)$; then $g$ is entire. If $n$ is even, we have $g(x) \equiv 0$ for all $x \in \mathbb{C}$ satisfying $|x| \geq \mathcal{R}(n, q)$. If $n$ is odd, then $(\eta_{q} g)(x) \equiv 0$ when $|x|$ is sufficiently large. Let $t = \eta_{q} x$; then $g(t) \equiv 0$ for all $|t| \geq \mathcal{R}'(n, q)$. This implies that the entire function $g(x)$ is bounded when $|x|$ is large enough, and by Liouville's Theorem, we have $g(x) = a_{0} f_{0}(x) + \cdots + a_{n} f_{n}(x) \equiv 0$ for all $x \in \mathbb{C}$, which contradicts the assumption that $f_{0}, \cdots, f_{n}$ are linearly independent. Thus, the assertion follows.
\end{proof}

\section{A general form of SMT for Askey-Wilson operator}\label{sec5}

As mentioned in the introduction, the general form of SMT plays a crucial role in certain problems. Here, we present the following Askey-Wilson version of the general form of SMT.

\smallskip

\begin{theorem}[Askey-Wilson version of the general form of SMT]\label{the5.1}
    Let $H_{1}, \ldots, H_{p}$ be arbitrary hyperplanes in $\mathbb{P}^{n}(\mathbb{C})$ with defining linear forms $L_{1}, \ldots, L_{p}$. Let $f:\mathbb{C} \to \mathbb{P}^{n}(\mathbb{C})$ be a holomorphic map with a reduced representation $\mathbf{f} = (f_{0}, \ldots, f_{n})$, where $f_{0}, \ldots, f_{n}$ are linearly independent, and let
    \begin{eqnarray}\label{eq5.1}
        \limsup_{r \to \infty}\frac{\log^{+}\log^{+} T_{f}(r)}{\log\log r} < 1.
    \end{eqnarray}
    Then,
    \begin{equation}\label{eq5.2}
    \int_{0}^{2\pi} \max_{\mathcal{K}} \sum_{j \in \mathcal{K}} \log \frac{\lVert \mathbf{f}(re^{i\theta}) \rVert}{|L_{j}(\mathbf{f})(re^{i\theta})|} \frac{d\theta}{2\pi} 
    + N_{\mathcal{W}}(r, 0) 
    \leq (n+1)T_{f}(r) + S_{\log}(r, f),
    \end{equation}
    where $r$ tends to infinity outside of an exceptional set $F$ satisfying (\ref{eq3.3}), $\mathcal{W} = \mathcal{W}(f_{0}, \ldots, f_{n})$, $N_{\mathcal{W}}(r, 0)=N(r, 1/\mathcal{W})$ and the maximum is taken over all subsets $\mathcal{K}\subset\{1, \cdots, p\}$ such that the linear forms $L_{j}$, $j \in \mathcal{K}$, are linearly independent.
\end{theorem}

\smallskip

\begin{proof}

Without loss of generality, we may assume that $p\geq n+1$ and $\#\mathcal{K}=n+1$. Let $E$ be the set of all injective maps $\mu: \{0, 1, \ldots, n\} \to \{1, \ldots, p\}$ such that $L_{\mu(0)}, \ldots, L_{\mu(n)}$ are linearly independent. Then, there exists a constant $A_{\mu} \in \mathbb{C} \setminus \{0\}$ such that
\[
\mathcal{W}(L_{\mu(0)}(\mathbf{f}), \ldots, L_{\mu(n)}(\mathbf{f})) = A_{\mu} \cdot \mathcal{W}(f_{0}, \ldots, f_{n}).
\]
Furthermore, from Lemma \ref{le4.4.1}, we have $\mathcal{W}(f_{0}, \ldots, f_{n}) \not\equiv 0$. Thus, when $r = |x|$ is sufficiently large, by applying (iv) in Proposition \ref{pro4.3.1}, we obtain
    \begin{eqnarray}\label{aeq5.1.1}
        &&\int_{0}^{2\pi}\max\limits_{\mathcal{K}}\sum\limits_{j\in\mathcal{K}}\log\frac{\lVert\textbf{f}(re^{i\theta})\rVert}{|L_{j}(\textbf{f})(re^{i\theta})|}\frac{d\theta}{2\pi}\nonumber\\
        &=&\int_{0}^{2\pi}\max\limits_{\mu\in E}\sum_{j=0}^{n}\log\frac{\lVert\textbf{f}(re^{i\theta})\rVert}{|L_{\mu(j)}(\textbf{f}(re^{i\theta}))|}\frac{d\theta}{2\pi}\nonumber\\
        &=&\int_{0}^{2\pi}\log\left(\max\limits_{\mu\in E}\left\{\frac{\lVert\textbf{f}(re^{i\theta})\rVert^{n+1}}{\prod_{j=0}^{n}|L_{\mu(j)}(\textbf{f}(re^{i\theta}))|}\right\}\right)\frac{d\theta}{2\pi}\nonumber\\
        &\leq&\int_{0}^{2\pi}\log\left(\max\limits_{\mu\in E}\frac{|\mathcal{W}(L_{\mu(0)}(\textbf{f}), \cdots, L_{\mu(n)}(\textbf{f}))|}{|L_{\mu(0)}(\textbf{f})\cdots L_{\mu(n)}(\textbf{f})|}(re^{i\theta})\right)\frac{d\theta}{2\pi}\nonumber\\
        &&+\int_{0}^{2\pi}\log\frac{\lVert\textbf{f}(re^{i\theta})\rVert^{n+1}}{|\mathcal{W}(f_{0}, \cdots, f_{n})(re^{i\theta})|}\frac{d\theta}{2\pi}+O(1)\nonumber\\
        &\leq&\int_{0}^{2\pi}\log\left(\max_{\mu\in E}\frac{|[\eta_{q}^{n}L_{\mu(0))}(\textbf{f})]\cdots [\eta_{q}^{-n}L_{\mu(0)}(\textbf{f})]|}{|L_{\mu(0)}(\textbf{f})\cdots L_{\mu(0)}(\textbf{f})|}(re^{i\theta})\right)\frac{d\theta}{2\pi}\nonumber\\
        &&+\int_{0}^{2\pi}\log\left(\max_{\mu\in E}\frac{\left|\mathcal{W}\left(\mathcal{D}_{q}\frac{L_{\mu(1)}(\textbf{f})}{L_{\mu(0)}(\textbf{f})}, \cdots, \mathcal{D}_{q}\frac{L_{\mu(n)}(\textbf{f})}{L_{\mu(0)}(\textbf{f})}\right)\right|}{\left|\frac{L_{\mu(1)}(\textbf{f})}{L_{\mu(0)}(\textbf{f})}\cdots \frac{L_{\mu(n)}(\textbf{f})}{L_{\mu(0)}(\textbf{f})}\right|}(re^{i\theta})\right)\frac{d\theta}{2\pi}\nonumber\\
        &&+\int_{0}^{2\pi}\log\frac{\lVert\textbf{f}(re^{i\theta})\rVert^{n+1}}{|\mathcal{W}(f_{0}, \cdots, f_{n})(re^{i\theta})|}\frac{d\theta}{2\pi}+O(1)\nonumber\\
        &=:&I_{1}+I_{2}+I_{3}+O(1).
    \end{eqnarray}
In the second inequality, we used the basic fact that 
\[
\log \max_{1 \leq k \leq m}\{a_{k}\} \leq \log \max_{1 \leq k \leq m}\left\{\frac{a_{k}}{b_{k}}\right\} + \log \max_{1 \leq k \leq m}\{b_{k}\}
\]
for $a_{k}, b_{k} \in \mathbb{R}^{+}$. Now, we first consider $I_{1}$. Suppose 
\[
\mathbf{g} = (g_{0}, \ldots, g_{n}) := (1, f_{1}/f_{0}, \ldots, f_{n}/f_{0}).
\]
Then, for any linear form $L$, we have 
$
L(\mathbf{f}) = f_{0}L(\mathbf{g}).
$
Hence,
    \begin{eqnarray}\label{aeq5.1.2}
        I_{1}&\leq&\int_{0}^{2\pi}\log\sum_{\mu\in E} \prod_{k=0}^{\lfloor n/2\rfloor}\left|\frac{(\eta_{q}^{n-2k}L_{\mu(0)}(\textbf{f}))(\eta_{q}^{-n+2k}L_{\mu(0)}(\textbf{f}))}{L_{\mu(0)}^{2}(\textbf{f})}\right|\frac{d\theta}{2\pi}\nonumber\\
        &=&\int_{0}^{2\pi}\log\prod_{k=0}^{\lfloor n/2\rfloor}\left|\frac{(\eta_{q}^{n-2k}f_{0})(\eta_{q}^{-n+2k}f_{0})}{f_{0}^{2}}\right|\frac{d\theta}{2\pi}\nonumber\\
        &&+\int_{0}^{2\pi}\log\sum_{\mu\in E}\prod_{k=0}^{\lfloor n/2\rfloor}\left|\frac{(\eta_{q}^{n-2k}L_{\mu(0)}(\textbf{g}))(\eta_{q}^{-n+2k}L_{\mu(0)}(\textbf{g}))}{L_{\mu(0)}^{2}(\textbf{g})}\right|\frac{d\theta}{2\pi}\nonumber\\
        &\leq&\sum_{k=0}^{\lfloor n/2\rfloor}\int_{0}^{2\pi}\log\left|\frac{(\eta_{q}^{n-2k}f_{0})(\eta_{q}^{-n+2k}f_{0})}{f_{0}^{2}}\right|\frac{d\theta}{2\pi}\nonumber\\
        &&+\sum_{\mu\in E}\sum_{k=0}^{\lfloor n/2\rfloor}\int_{0}^{2\pi}\log^{+}\left|\frac{(\eta_{q}^{n-2k}L_{\mu(0)}(\textbf{g}))(\eta_{q}^{-n+2k}L_{\mu(0)}(\textbf{g}))}{L_{\mu(0)}^{2}(\textbf{g})}\right|\frac{d\theta}{2\pi}+O(1)\nonumber\\
        &=:&I_{11}+I_{12}+O(1).
    \end{eqnarray}
Due to $\sigma_{2}^{\log}(f) < 1$, it follows from (\ref{equa2.2}) that 
\[
\limsup_{r \to \infty} \frac{\log^{+}\log^{+} N(r, 1/f_{0})}{\log\log r} < 1.
\]
Thus, we can obtain a similar estimation for $N(r, 1/f_{0})$ as in Lemma \ref{le3.3} using the same analysis. Furthermore, since 
\[
(\eta_{q}^{\pm(n-2k)}f_{0})(x) = (\eta_{q^{\pm(n-2k)}}f_{0})(x)
\]
holds for all sufficiently large $|x|$, and 
$
(\eta_{q^{n-2k}}f_{0})(\eta_{q^{-n+2k}}f_{0})
$
is entire for $k = 0, \ldots, \lfloor n/2 \rfloor$, by applying the First Main Theorem to 
$
(\eta_{q^{n-2k}}f_{0})(\eta_{q^{-n+2k}}f_{0}),
$
we obtain
    \begin{eqnarray}\label{aeq5.3}
        I_{11}&=&\sum_{k=0}^{\lfloor n/2\rfloor}\left(N\left(r, \frac{1}{(\eta_{q^{n-2k}}f_{0})(\eta_{q^{-n+2k}}f_{0})}\right)-2N\left(r, \frac{1}{f_{0}}\right)+O(\log r)\right)\nonumber\\
        &=&\sum_{k=0}^{\lfloor n/2\rfloor}\left(N\left(r, \frac{1}{(\eta_{q}^{n-2k}f_{0})(\eta_{q}^{-n+2k}f_{0})}\right)-2N\left(r, \frac{1}{f_{0}}\right)+O(\log r)\right)\nonumber\\
        &=&O\left(\frac{N(r, 1/f_{0})}{\log\log r}\right)+O(\log r)= S_{\log}(r, f),
    \end{eqnarray}
    where $r$ tends to infinity outside of an exceptional set $E$ satisfying (\ref{eq3.3}). For $I_{12}$, since for any linear form $L$ we have
\begin{eqnarray}\label{aeq5.11}
    T(r, L(\mathbf{g})) \leq K \sum_{i=1}^{n} T(r, g_{i}) + O(1) = O(T_{f}(r)),
\end{eqnarray}
where the constant $K > 0$ is independent of $r$ and $\mathbf{g}$, it follows that, by applying Lemma \ref{le3.2}, we have
\begin{eqnarray}\label{aeq5.4}
    I_{12} = O\left(\sum_{\mu \in E} \frac{T(r, L_{\mu(0)}(\mathbf{g}))}{\log\log r}\right) + O(\log r) = S_{\log}(r, f),
\end{eqnarray}
where $r$ tends to infinity outside of an exceptional set $E$ satisfying (\ref{eq3.3}). It is important to note that 
\[
\log \max\limits_{1 \leq j \leq n} \{a_{j}\} \leq \sum_{j=1}^{n} \log a_{j}
\]
for positive numbers $a_{j}$ may not hold, but 
\[
\log \max\limits_{1 \leq j \leq n} \{a_{j}\} \leq \sum_{j=1}^{n} \log^{+} a_{j}
\]
is valid.

For $I_{2}$, we define 
\[
h_{\mu_{i}} = \frac{L_{\mu(i)}(\mathbf{f})}{L_{\mu(0)}(\mathbf{f})} = \frac{L_{\mu(i)}(\mathbf{g})}{L_{\mu(0)}(\mathbf{g})}, \quad \text{for } i = 1, \ldots, n,
\]
where $\mathbf{g} := (1, f_{1}/f_{0}, \ldots, f_{n}/f_{0})$. Thus, from (\ref{aeq5.11}), we have 
\[
T(r, h_{\mu_{i}}) \leq T(r, L_{\mu(i)}(\mathbf{g})) + T(r, L_{\mu(0)}(\mathbf{g})) + O(1) = O(T_{f}(r)).
\]
It follows from (\ref{eq4.1.2}), Lemma \ref{le3.2} and Theorem \ref{the3.4} that
    \begin{eqnarray}\label{aeq5.5}
        I_{2}&\leq&\int_{0}^{2\pi}\sum_{\mu\in E}\log ^{+}\frac{\left|\mathcal{W}(\mathcal{D}_{q}h_{\mu_{1}}, \cdots, \mathcal{D}_{q}h_{\mu_{n}})\right|}{|h_{\mu_{1}}\cdots h_{\mu_{n}}|}(re^{i\theta})\frac{d\theta}{2\pi}\nonumber\\
        &\leq&\int_{0}^{2\pi}\sum_{\mu\in E}\sum_{i=1}^{n}\sum_{k=0}^{n-1}\log^{+}\left|\frac{\mathcal{A}_{q^{n-1-k}}\mathcal{D}_{q}^{k}h_{\mu_{i}}}{h_{\mu_{i}}}\right|(re^{i\theta})\frac{d\theta}{2\pi}+O(1)\nonumber\\
        &=&S_{\log}(r, f),
    \end{eqnarray}
     where $r$ tends infinity outside of an exceptional set $E$ satisfying (\ref{eq3.3}). Lastly, by applying the First Main Theorem to $\mathcal{W}(f_{0}, \cdots, f_{n})$, we have
    \begin{eqnarray}\label{aeq5.6}
        I_{3}=(n+1)T_{f}(r)-N_{\mathcal{W}}(r, 0)+O(1).
    \end{eqnarray}
    Then we complete the proof by combining (\ref{aeq5.1.1}), (\ref{aeq5.1.2}), (\ref{aeq5.3}), (\ref{aeq5.4}), (\ref{aeq5.5}) and (\ref{aeq5.6}).
    \end{proof}

Recall that the classical counting function truncated by $M\in\mathbb{N}$ for any meromorphic function $f$ is defined by
\begin{eqnarray}
    \overline{n}^{(M)}\left(r, \frac{1}{f-a}\right) &=& \sum_{x \in D(0, r)} \left(\nu_{f-a}^{0}(x) - \min\limits_{0\leq t\leq M}\left\{\nu_{(f-a)^{(t)}}^{0}(x)\right\}\right)\nonumber\\
    &=&\sum_{x \in D(0, r)}\min\left\{\nu_{f-a}^{0}(x), M\right\}\label{aeq5.6.1.1.1}
\end{eqnarray}
where $\nu_{f}^{0}(x) \geq 0$ denotes the order of zero of $f$ at $x$, and $a \in \mathbb{C}$.

Furthermore, the Askey-Wilson-type truncated counting function for any meromorphic function $f$ is defined by
\begin{eqnarray}
    \tilde{n}_{AW}\left(r, \frac{1}{f-a}\right) &=& \sum_{x \in D(0, r)} \left(\nu_{f-a}^{0}(x) - \min\left\{\nu_{f-a}^{0}(x), \nu_{\eta_{q}\mathcal{D}_{q}(f-a)}^{0}(x)\right\}\right)\nonumber\\
    &=& \sum_{x \in D(0, r)} \left(\nu_{f-a}^{0}(x) - \min\left\{\nu_{f-a}^{0}(x), \nu_{\eta_{q}^{2}(f-a)}^{0}(x)\right\}\right)+O(1)\label{aeq5.5.1}
\end{eqnarray}
which can be seen in \cite[Definitions (110) and (112)]{NADV}. Similarly, $\tilde{n}_{AW}(r, f)$ and the corresponding Askey-Wilson-type truncated integrated counting functions $\tilde{N}_{AW}(r, 1/(f-a))$ and $\tilde{N}_{AW}(r, f)$ can be defined \cite[Definitions (111) and (113)]{NADV}. Then Chiang and Feng formulated the truncated form of the Askey-Wilson-type Second Main Theorem \cite[Theorem 7.1]{NADV}.

In this paper, we redefine the \textit{Askey-Wilson-type truncated counting function} using a form-similar with (\ref{aeq5.6.1.1.1}) approach that is distinct from Chiang and Feng's definition, and generalize the Askey-Wilson-type Truncated Second Main Theorem in projective space. Specifically, let $M \in \mathbb{N}$. For any meromorphic function $f$ and $a \in \mathbb{C}$, we define the \textit{Askey-Wilson version of counting function truncated by $M$} as
\begin{eqnarray}\label{aeq5.5.2}
    \tilde{n}_{AW}^{[M]}\left(r, \frac{1}{f-a}\right) = \sum_{x \in D(0, r)} \left(\nu_{\eta_{q}^{\delta(M)}(f-a)}^{0}(x) - \min\limits_{0 \leq t \leq M} \left\{\nu_{\mathcal{A}_{q^{M-t}}\mathcal{D}_{q}^{t}(f-a)}^{0}(x)\right\}\right),
\end{eqnarray}
where
\begin{eqnarray}\label{aeq5.5.3}
    \delta(M) = \begin{cases} 
        -1, & M \text{ is odd}, \\
        0, & M \text{ is even}.
    \end{cases}
\end{eqnarray}
The \textit{integrated counting function truncated by $M$} is denoted by
\begin{eqnarray*}
    \tilde{N}_{AW}^{[M]}\left(r, \frac{1}{f-a}\right) &=& \int_{0}^{r} \frac{\tilde{n}_{AW}^{[M]}\left(t, 1/(f-a)\right) - \tilde{n}_{AW}^{[M]}\left(0, 1/(f-a)\right)}{t} dt \\
    && + \tilde{n}_{AW}^{[M]}\left(0, 1/(f-a)\right) \log r.
\end{eqnarray*}

Similarly, we can define 
\[
\tilde{n}_{AW}^{[M]}\left(r, f\right) = \tilde{n}_{AW}^{[M]}\left(r, 1/(1/f)\right) \quad \text{and} \quad 
\tilde{N}_{AW}^{[M]}\left(r, f\right) = \tilde{N}_{AW}^{[M]}\left(r, 1/(1/f)\right).
\]

\noindent\textbf{Remark}: By comparing (\ref{aeq5.6.1.1.1}) and (\ref{aeq5.5.2}), we can see that the term $\mathcal{A}_{q^{M-t}}\mathcal{D}_{q}^{t}f$ plays the role with the derivative $f^{(t)}$. In \cite{Rad}, Ishizaki, the second author, Li, and Tohge defined the difference analogue of the truncated counting function and the corresponding integrated counting function with this method. In their definition, $\overline{f}^{[t]}$ plays the role with $f^{(t)}$.

As for $\delta(M) \in \{M-2t\}_{t=0}^{M}$, from the following lemma, we will see that if there are infinitely many $a$-points of $f$, then $\delta(M)$ ensures $\tilde{n}_{AW}^{[M]}(r, 1/(f-a))$ increases as $r$ becomes sufficiently large. Before stating the lemma, we provide some explanations.

Let $f$ and $g$ be two complex functions that are analytic at $x_{0}$. Then, we can easily verify:
\begin{eqnarray}\label{aeq5.2.1}
    \nu_{f-g}^{0}(x_{0}) \geq \min\left\{\nu_{f}^{0}(x_{0}), \nu_{g}^{0}(x_{0})\right\},
\end{eqnarray}
and
\begin{eqnarray}\label{aeq5.2.2}
    \min\left\{\nu_{f}^{0}(x_{0}), \nu_{f-g}^{0}(x_{0})\right\} = \min\left\{\nu_{f}^{0}(x_{0}), \nu_{g}^{0}(x_{0})\right\}.
\end{eqnarray}
We say there is no cancellation for $f-g$ at $x_{0}$ when equality holds in (\ref{aeq5.2.1}), and there is cancellation for $f-g$ at $x_{0}$ when equality in (\ref{aeq5.2.1}) does not hold.

\smallskip

\begin{lemma}\label{le5.3}
    Let $M \in \mathbb{N}$ and $f$ be a meromorphic function. Then, we have
    \begin{eqnarray}\label{eq5.3}
        \min\limits_{0 \leq t \leq M}\left\{\nu_{\mathcal{A}_{q^{M-t}}\mathcal{D}_{q}^{t}f}^{0}(x)\right\} = \min\limits_{0 \leq t \leq M}\left\{\nu_{\eta_{q}^{M-2t}f}^{0}(x)\right\}
    \end{eqnarray}
    for all $x \in \mathbb{C}$ satisfying $|x| > \mathcal{R}(M, q)$.
\end{lemma}

\smallskip

\begin{proof}
    We first prove that
    \begin{eqnarray}\label{aeq5.3.1}
        &&\min\left\{\nu_{\eta_{q}^{M}f + \eta_{q}^{-M}f}^{0}(x), \min\limits_{0 \leq t \leq M-1}\left\{\nu_{\eta_{q}^{M-2t}f - \eta_{q}^{M-2t-2}f}^{0}(x)\right\}\right\} \nonumber \\
        &=& \min\limits_{0 \leq t \leq M}\left\{\nu_{\eta_{q}^{M-2t}f}^{0}(x)\right\},
    \end{eqnarray}
    for all $x \in \mathbb{C}$ satisfying $|x| > \mathcal{R}(M, q)$. Indeed, if there is cancellation for all $\eta_{q}^{M-2t}f(x) - \eta_{q}^{M-2t-2}f(x)$, $t = 0, \ldots, M-1$, then there is no cancellation for $\eta_{q}^{M}f(x) + \eta_{q}^{-M}f(x)$. Thus, (\ref{aeq5.3.1}) holds and equals $\nu_{\eta_{q}^{M-2t}f}^{0}(x)$ for $t = 0, \ldots, M$. 

    If there exist some $t \in \{0, 1, \ldots, M-1\}$ such that there is no cancellation for $\eta_{q}^{M-2t}f(x) - \eta_{q}^{M-2t-2}f(x)$, we claim that
    \begin{eqnarray}\label{aeq5.3.1.1}
        \min\limits_{0 \leq t \leq M-1}\left\{\nu_{\eta_{q}^{M-2t}f - \eta_{q}^{M-2t-2}f}^{0}(x)\right\} = \min\limits_{0 \leq t \leq M}\left\{\nu_{\eta_{q}^{M-2t}f}^{0}(x)\right\}.
    \end{eqnarray}
    Then $\nu_{\eta_{q}^{M}f + \eta_{q}^{-M}f}^{0}(x) \geq \min\left\{\nu_{\eta_{q}^{M}f}^{0}(x), \nu_{\eta_{q}^{-M}f}^{0}(x)\right\} \geq \min\limits_{0 \leq t \leq M}\left\{\nu_{\eta_{q}^{M-2t}f}^{0}(x)\right\}$, which proves (\ref{aeq5.3.1}). Without loss of generality, assume $$\nu_{\eta_{q}^{M}f - \eta_{q}^{M-2}f}^{0}(x) = \min\left\{\nu_{\eta_{q}^{M}f}^{0}(x), \nu_{\eta_{q}^{M-2}f}^{0}(x)\right\}.$$ It follows from (\ref{aeq5.2.2}) that
    \begin{eqnarray*}
        &&\min\limits_{0 \leq t \leq M-1}\left\{\nu_{\eta_{q}^{M-2t}f - \eta_{q}^{M-2t-2}f}^{0}(x)\right\} \\
        &=& \min\left\{\nu_{\eta_{q}^{M}f}^{0}(x), \min\left\{\nu_{\eta_{q}^{M-2}f}^{0}(x), \nu_{\eta_{q}^{M-2}f - \eta_{q}^{M-4}f}^{0}(x)\right\}, \ldots\right\} \\
        &=& \min\left\{\nu_{\eta_{q}^{M}f}^{0}(x), \nu_{\eta_{q}^{M-2}f}^{0}(x), \nu_{\eta_{q}^{M-4}f}^{0}(x), \ldots\right\} \\
        &=& \cdots \,\,=\min\limits_{0 \leq t \leq M}\left\{\nu_{\eta_{q}^{M-2t}f}^{0}(x)\right\},
    \end{eqnarray*}
    which proves the claim (\ref{aeq5.3.1.1}).

    Now, we prove (\ref{eq5.3}) by induction. When $M = 1$, since at most one of $\eta_{q}f + \eta_{q}^{-1}f$ and $\eta_{q}f - \eta_{q}^{-1}f$ have cancellation, it follows from (\ref{aeq5.2.1}) that
    \begin{eqnarray*}
        \min\left\{\nu_{\mathcal{A}_{q}f}^{0}(x), \nu_{\mathcal{D}_{q}f}^{0}(x)\right\} &=& \min\left\{\nu_{\eta_{q}f + \eta_{q}^{-1}f}^{0}(x), \nu_{\eta_{q}f - \eta_{q}^{-1}f}^{0}(x)\right\} \\
        &=& \min\left\{\nu_{\eta_{q}f}^{0}(x), \nu_{\eta_{q}^{-1}f}^{0}(x)\right\}.
    \end{eqnarray*}
    Thus, the assertion is valid for $M = 1$. Assume (\ref{eq5.3}) holds for $M-1$. Then for $M$, we have
    \begin{eqnarray*}
            &&\min\limits_{0\leq t\leq M}\left\{\nu_{\mathcal{A}_{q^{M-t}}\mathcal{D}_{q}^{t}f}^{0}(x)\right\}\\
            &=&\min\left\{\nu_{\mathcal{A}_{q^{M}}f}^{0}(x), \min\limits_{1\leq t\leq M}\left\{\nu_{\mathcal{A}_{q^{M-t}}\mathcal{D}_{q}^{t}f}^{0}(x)\right\}\right\}\\
            &=&\min\left\{\nu_{\mathcal{A}_{q^{M}}f}^{0}(x), \min\limits_{0\leq t\leq M-1}\left\{\nu_{\mathcal{A}_{q^{M-1-t}}\mathcal{D}_{q}^{t}(\mathcal{D}_{q}f)}^{0}(x)\right\}\right\}\\
            &=&\min\left\{\nu_{\mathcal{A}_{q^{M}}f}^{0}(x), \min\limits_{0\leq t\leq M-1}\left\{\nu_{\eta_{q}^{M-1-2t}(\mathcal{D}_{q}f)}^{0}(x)\right\}\right\}\\
            &=&\min\left\{\nu_{\eta_{q}^{M}f+\eta_{q}^{-M}f}^{0}(x), \min\limits_{0\leq t\leq M-1}\left\{\nu_{\eta_{q}^{M-2t}f-\eta_{q}^{M-2t-2}f}^{0}(x)\right\}\right\}\\
            &=&\min\limits_{0\leq j\leq M}\left\{\nu_{\eta_{q}^{M-2j}f}^{0}(x)\right\},
        \end{eqnarray*}
    where the third equality follows by the induction hypothesis and the last equality is obtained from (\ref{aeq5.3.1}). Hence, the proof is complete.
\end{proof}

From this lemma, we also have
\begin{eqnarray*}
    \tilde{n}_{AW}^{[1]}\left(r, \frac{1}{f-a}\right) = \sum_{x \in D(0, r)} \left(\nu_{\eta_{q}^{-1}(f-a)}^{0}(x) - \min \left\{\nu_{\eta_{q}(f-a)}^{0}(x), \nu_{\eta_{q}^{-1}(f-a)}^{0}(x)\right\}\right)+O(1).
\end{eqnarray*}
Hence there is only one $\eta_{q}^{-1}$ difference between definitions (\ref{aeq5.5.1}) and (\ref{aeq5.5.2}), up to a constant, and the difference can be ignored to some degree from the following theorem. This theorem also implies that $\tilde{N}_{AW}(r, 1/(f-a))$ can be replaced by $\tilde{N}_{AW}^{[1]}(r, 1/(f-a))$ in \cite[Theorem 7.1]{NADV}.

\smallskip

\begin{theorem}\label{the5.2}
    Let $f$ be a non-constant meromorphic function with $\sigma_{2}^{\log}(f) < 1$. Then, for each $a \in \hat{\mathbb{C}}$, we have
    \begin{eqnarray}\label{eq5.2.00}
        \tilde{N}_{AW}^{[1]}\left(r, \frac{1}{f-a}\right) = \tilde{N}_{AW}\left(r, \frac{1}{f-a}\right) + S_{\log}(r, f),
    \end{eqnarray}
    where $r$ tends to infinity outside of a set $F$ satisfying (\ref{eq3.3}).
\end{theorem}

\smallskip

\begin{proof}
There is nothing to prove when there are only finitely many $a$-points for $f$, because $\tilde{N}_{AW}^{[1]}(r, 1/(f-a)) = O(\log r)$ and $\tilde{N}_{AW}(r, 1/(f-a)) = O(\log r)$, both of which are included in the error term $S_{\log}(r, f)$. Hence, we assume that there are infinitely many $a$-points for $f$. By following Lemma \ref{le5.3}, when $r$ is large enough, we have
\begin{eqnarray*}
    &&\tilde{n}_{AW}^{[1]}\left(r, \frac{1}{\eta_{q}(f-a)}\right) \\
    &=& \sum_{x \in D(0, r)} \left(\nu_{f-a}^{0}(x) - \min\left\{\nu_{f-a}^{0}(x), \nu_{\eta_{q}^{2}(f-a)}^{0}(x)\right\}\right) + O(1) \\
    &=& \tilde{n}_{AW}\left(r, \frac{1}{f-a}\right) + O(1),
\end{eqnarray*}
which implies
\begin{eqnarray}\label{aeq5.2.1.1}
    \tilde{N}_{AW}^{[1]}\left(r, \frac{1}{\eta_{q}(f-a)}\right) = \tilde{N}_{AW}\left(r, \frac{1}{f-a}\right) + O(\log r).
\end{eqnarray}

Since $\tilde{n}_{AW}^{[1]}(r, 1/(f-a))$ and, consequently, $\tilde{N}_{AW}^{[1]}(r, 1/(f-a))$ are non-decreasing when $r$ is sufficiently large (by definition), and from Lemma \ref{le3.3}, we have
\begin{eqnarray*}
    \tilde{N}_{AW}^{[1]}(r, 1/(f-a)) &\leq& N(r, 1/\eta_{q}^{-1}(f-a)) \\
&=& N(r, 1/(f-a)) + S_{\log}(r, f) \\
&\leq& T(r, f) + S_{\log}(r, f),
\end{eqnarray*}
it follows that
\[
\limsup_{r \to \infty} \frac{\log\log \tilde{N}_{AW}^{[1]}(r, 1/(f-a))}{\log\log r} < 1.
\]
Thus, we can apply the proof of Lemma \ref{le3.3} to $\tilde{N}_{AW}(r, 1/(f-a))$ and obtain a similar result:
\begin{eqnarray}\label{aeq5.2.1.2}
    \tilde{N}_{AW}^{[1]}\left(r, \frac{1}{\eta_{q}(f-a)}\right) = \tilde{N}_{AW}^{[1]}\left(r, \frac{1}{f-a}\right) + S_{\log}(r, f).
\end{eqnarray}

The assertion then follows by combining (\ref{aeq5.2.1.1}) and (\ref{aeq5.2.1.2}).
\end{proof}

Next, we will extend the Truncated Second Main Theorem for the Askey-Wilson operator in projective space.

\smallskip

\begin{theorem}[Askey-Wilson version of the truncated SMT]\label{the5.4}
    Let $H_{1}, \ldots, H_{p}$ be arbitrary hyperplanes in $\mathbb{P}^{n}(\mathbb{C})$ in general position with defining linear forms $L_{1}, \ldots, L_{p}$. Let $f: \mathbb{C} \to \mathbb{P}^{n}(\mathbb{C})$ be a holomorphic map with a reduced representation $\mathbf{f} = (f_{0}, \ldots, f_{n})$ such that $f_{0}, \ldots, f_{n}$ are linearly independent, and let
    \begin{eqnarray}\label{eq5.4.1}
        \limsup_{r \to \infty}\frac{\log^{+}\log^{+} T_{f}(r)}{\log\log r} < 1.
    \end{eqnarray}
    Then,
    \begin{eqnarray}
        (p - (n+1))T_{f}(r) &\leq& \sum_{j=1}^{p} N\left(r, \frac{1}{L_{j}(\mathbf{f})}\right) - N_{\mathcal{W}}\left(r, 0\right) + S_{\log}(r, f), \label{eq5.4.2} \\
        &\leq& \sum_{j=1}^{p} \tilde{N}_{AW}^{[n]}\left(r, \frac{1}{L_{j}(\mathbf{f})}\right) + S_{\log}(r, f), \label{eq5.4.3}
    \end{eqnarray}
    where $r$ tends to infinity outside of an exceptional set $E$ satisfying (\ref{eq3.3}), $\mathcal{W} = \mathcal{W}(f_{0}, \ldots, f_{n})$ and $N_{\mathcal{W}}(r, 0)=N(r, 1/\mathcal{W})$.
\end{theorem}

\smallskip

\begin{proof}
For any $x\in\mathbb{C}$, we may assume
\begin{eqnarray*}
    |L_{1(x)}(\textbf{f})(x)|\leq|L_{2(x)}(\textbf{f})(x)|\leq\cdots \leq |L_{p(x)}(\textbf{f})(x)|.
\end{eqnarray*}
Since $H_{1}, \ldots, H_{p}$ are in general position, then there exists a constant $A>0$ only depending on the coefficients of $L_{j}$, $j=1, \cdots, p$ such that $\lVert \textbf{f}(x) \rVert\leq A\cdot |L_{(n+1)(x)}(\textbf{f})(x)|$. Then by applying Theorem \ref{the5.1}, we have
\begin{eqnarray*}
    \sum_{j=1}^{p}\int_{0}^{2\pi}\log\frac{\|\mathbf{f}(re^{i\theta})\|}{|L_{j}(\mathbf{f})(re^{i\theta})|}\frac{d\theta}{2\pi} + N_{\mathcal{W}}(r, 0) \leq (n+1)T_{f}(r) + S_{\log}(r, f),
\end{eqnarray*}
where $r$ tends to infinity outside of an exceptional set $F$ satisfying (\ref{eq3.3}). Since $L_{j}(\mathbf{f})$ is entire for $j = 1, \ldots, p$, by the First Main Theorem, we obtain
\begin{eqnarray*}
    \int_{0}^{2\pi}\log|L_{j}(\mathbf{f})(re^{i\theta})|\frac{d\theta}{2\pi} = N\left(r, \frac{1}{L_{j}(\mathbf{f})}\right) + O(1).
\end{eqnarray*}
Hence, the first inequality (\ref{eq5.4.2}) is valid.

For the second inequality (\ref{eq5.4.3}), we first claim
\begin{eqnarray}\label{aeq5.4.1}
    \sum_{j=1}^{p}N\left(r, \frac{1}{\eta_{q}^{\delta(n)}L_{j}(\mathbf{f})}\right) - N_{\mathcal{W}}(r, 0) \leq \sum_{j=1}^{p}\tilde{N}_{AW}^{[n]}\left(r, \frac{1}{L_{j}(\mathbf{f})}\right) + S_{\log}(r, f),
\end{eqnarray}
where $\delta(n)$ is defined as in (\ref{aeq5.5.3}). Since $H_{1}, \ldots, H_{p}$ are in general position, at most $n$ hyperplanes pass through $f$ for any fixed $x \in \mathbb{C}$. Without loss of generality, we may assume
\begin{eqnarray*}
    \nu_{\eta_{q}^{\delta(n)}L_{1}(\mathbf{f})}^{0}(x) \geq \cdots \geq \nu_{\eta_{q}^{\delta(n)}L_{k}(\mathbf{f})}^{0}(x) > 0 = \nu_{\eta_{q}^{\delta(n)}L_{k+1}(\mathbf{f})}^{0}(x) = \cdots = \nu_{\eta_{q}^{\delta(n)}L_{p}(\mathbf{f})}^{0}(x),
\end{eqnarray*}
where $0 \leq k \leq n$. Furthermore, we have
\begin{eqnarray*}
    \mathcal{W}(f_{0}, \ldots, f_{n})(x) = B \cdot \mathcal{W}(L_{1}(\mathbf{f}), \ldots, L_{n+1}(\mathbf{f}))(x),
\end{eqnarray*}
where $B$ is a non-zero constant. Then,
\begin{eqnarray*}
    \nu_{\mathcal{W}}^{0}(x) = \nu_{\mathcal{W}(L_{1}(\mathbf{f}), \ldots, L_{n+1}(\mathbf{f}))}^{0}(x) \geq \sum_{j=1}^{p} \min\limits_{0 \leq t \leq n}\left\{\nu_{\mathcal{A}_{q^{n-t}}\mathcal{D}_{q}^{t}L_{j}(\mathbf{f})}^{0}(x)\right\}
\end{eqnarray*}
for all $x \in \mathbb{C}$ satisfying $|x| > \mathcal{R}(n, q)$. By the definition of $\tilde{N}_{AW}^{[1]}(r, 1/(L_{j}(\mathbf{f})))$, we can obtain the claim (\ref{aeq5.4.1}).

Let $\mathbf{g} = (1, f_{1}/f_{0}, \ldots, f_{n}/f_{0})$. Then $L_{j}(\mathbf{f}) = f_{0}L_{j}(\mathbf{g}) \not\equiv 0$ because $f_{0}, \ldots, f_{n}$ are linearly independent. Hence, from Lemma \ref{le3.1}, we have
\begin{eqnarray*}
    N\left(r, \frac{1}{L_{j}(\mathbf{f})}\right) &\leq& N\left(r, \frac{1}{f_{0}}\right) + N\left(r, \frac{1}{L_{j}(\mathbf{g})}\right) \\
    &\leq& T_{f}(r) + T(r, L_{j}(\mathbf{g})) + O(1) \\
    &=& K T_{f}(r),
\end{eqnarray*}
where $j = 1, \ldots, p$ and $K$ is a finite positive number depending only on the linear form $L_{j}$. By a similar analysis to Lemma \ref{le3.3}, we have
\begin{eqnarray}\label{aeq5.4.2}
    \sum_{j=1}^{p}N\left(r, \frac{1}{\eta_{q}^{\delta(n)}L_{j}(\mathbf{f})}\right) = \sum_{j=1}^{p}N\left(r, \frac{1}{L_{j}(\mathbf{f})}\right) + S_{\log}(r, f).
\end{eqnarray}
Thus, (\ref{eq5.4.3}) follows by combining (\ref{aeq5.4.1}) and (\ref{aeq5.4.2}).
\end{proof}

\section{Askey-Wilson version of SMT with hypersurfaces}\label{sec6}

    As far as we know, Shi and Yan were the first to consider the number of irreducible components of hypersurfaces to establish a new SMT for holomorphic maps into projective space \cite{Shi}. However, in their key lemma (\cite[Lemma 2.4]{Shi}), the reorganization of the homogeneous polynomial is restricted to having only two elements. In this paper, we extend their key lemma and generalize the corresponding SMT for the Askey-Wilson operator. First, we introduce a new definition for a homogeneous polynomial.

    \smallskip

\begin{definition}[Min-max $s^{'}$-th decomposition]\label{def6.1} 
    Let $Q$ be a homogeneous polynomial in $\overline{k}[x_{0}, \ldots, x_{n}]$ of degree $d$, and suppose it has $s$ distinct irreducible factors. Fix a positive integer $s^{'} \leq s$, and let 
    \[
\begin{aligned}
    \Omega_{Q}(s^{'}) = \big\{ &\{R_{1}, \ldots, R_{s^{'}}\} :  \; Q = R_{1}\cdots R_{s^{'}}, \\
    & R_{j} \text{ is a non-constant homogeneous polynomial in } \overline{k}[x_{0}, \ldots, x_{n}], \\
    & j = 1, \ldots, s^{'}, \text{ and there are no common factors between any two } R_{j} \big\}.
\end{aligned}
\]
    We call a decomposition $\mathcal{R}^{*} = \{R_{1}^{*}, \ldots, R_{s^{'}}^{*}\} \in \Omega_{Q}(s^{'})$ the \textit{min-max $s^{'}$-th decomposition of $Q$} if
    \begin{eqnarray*}
        \max\limits_{1 \leq j \leq s^{'}} \deg R_{j}^{*} = \min\limits_{\mathcal{R} \in \Omega_{Q}(s^{'})} \max\limits_{R_{j} \in \mathcal{R}} \deg R_{j}.
    \end{eqnarray*}
    We also define $d_{Q}^{*}(s^{'}) = \max\limits_{1 \leq j \leq s^{'}} \deg R_{j}^{*}$.
\end{definition}

\smallskip

From now on, we adopt the notation $\lceil x \rceil$ ($\lfloor x\rfloor$) to express the smallest (largest) integer greater (less) than or equal to $x\in\mathbb{R}$.

The following lemma is a generalization of the key lemma in \cite[Lemma 2.4]{Shi} (for $s^{'}=2$), while the proof is completely different.

\smallskip

\begin{lemma}\label{le6.2}
    Under the assumptions of Definition \ref{def6.1}, we have
    \begin{eqnarray}\label{eq6.2}
        d_{Q}^{*}(s^{'}) \leq \max\left\{d - s + 1, \left\lceil\frac{d}{s^{'}}\right\rceil\right\}.
    \end{eqnarray}
\end{lemma}

\smallskip

    \begin{proof}
Let $Q = Q_{1} \cdots Q_{s}$, where $Q_{i}$ is a non-constant factor of $Q$ such that there are no common factors between any two $Q_{i}$, and $\deg Q_{i} = d_{i}$ for $i = 1, \ldots, s$. Without loss of generality, let $d_{1} \geq d_{2} \geq \cdots \geq d_{s} \geq 1$. The assertion clearly holds when $s^{'} = 1$ or $s = s^{'}$. Hence, we will assume $s > s^{'} \geq 2$. We first define an algorithm to distribute $d_{i}$.

\smallskip

\textbf{Step 1}: Let $E_{j} = \{j\}$ and $R_{j} = Q_{j}$ for $j = 1, \ldots, s^{'}$, then $\deg R_{1} = d_{1}, \ldots, \deg R_{s^{'}} = d_{s^{'}}$.

\textbf{Step 2}: Suppose $d_{1}, \ldots, d_{k-1}$ have been used, where $k \in \mathbb{N}$ (i.e., $\{1, \ldots, k-1\} = \cup_{j=1}^{s^{'}} E_{j}$). Now, we proceed to distribute $d_{k}$. Denote $j^{*} \in \{1, \ldots, s^{'}\}$ such that 
\[
\deg R_{j^{*}} = \min\limits_{1 \leq j \leq s^{'}} \deg R_{j}.
\]
(There may be several $j^{*}$ for which the equality holds, but any one of them can be chosen.) Let $E_{j^{*}} = E_{j^{*}} \cup \{k\}$ and $R_{j^{*}} = R_{j^{*}} Q_{k}$, then 
\[
\deg R_{j^{*}} = \deg R_{j^{*}} + d_{k}.
\]
Repeat this step until all $d_{i}$, $i = 1, \ldots, s$, are exhausted.

        \smallskip

It is clear that any two of $R_{j}$ do not have common factors; hence,
\begin{eqnarray*}
    \{R_{1}, \ldots, R_{s^{'}}\} \in \Omega_{Q}(s^{'}).
\end{eqnarray*}
Therefore,
\begin{eqnarray*}
    \max_{1 \leq j \leq s^{'}} \deg R_{j}^{*} \leq \max_{1 \leq j \leq s^{'}} \deg R_{j}.
\end{eqnarray*}
Thus, it is sufficient to estimate the upper bound for $\max_{1 \leq j \leq s^{'}} \deg R_{j}$. Note that the larger $d_{i}$ values are distributed first in this algorithm. Suppose $d_{1} = N$, then $\{d_{i}\}_{i=1}^{s} \subset \{1, \ldots, N\}$. 

We say the decomposition is \textit{in the $k$-th stage} ($k \in \{1, \ldots, N\}$) if we have just used all $d_{i}$ values equal to $k$. Denote by $d^{(k)}$ and $s^{(k)}$ be the summation of $d_{i}$ and their numbers we have used in $k$ stage. Define $I_{\max}^{(k)}$ and $I_{\min}^{(k)}$ as the maximum and minimum of the degrees of $R_{j}$ for $j = 1, \ldots, s^{'}$ in the $k$-th stage, respectively. That is
\begin{eqnarray*}
    d^{(k)}&:=&\sum_{j=1}^{s^{'}}\deg R_{j}=\sum_{d_{i}\geq k}d_{i},\\
    s^{(k)}&:=&\#\{d_{i}\geq k\},\\
    I_{\max}^{(k)}&:=&\max_{1\leq j\leq s^{'}}\deg R_{j},\\
    I_{\min}^{(k)}&:=&\min_{1\leq j\leq s^{'}}\deg R_{j},\\
\end{eqnarray*}
where $\{R_{1}, \cdots, R_{s^{'}}\}$ is in the $k$-th stage. In particular, $d^{(1)}=d$ and $s^{(1)}=s$. We also adopt the notation $I_{\max}:=I_{\max}^{(1)}$ and $I_{\min}:=I_{\min}^{(1)}$. (A specific example following this algorithm can be seen from Example \ref{ex1} below.)

Now we claim that
\begin{eqnarray}\label{aeq6.2.1}
    I_{\max}^{(k)} \leq d^{(k)} - s^{(k)} + 1,
\end{eqnarray}
for $k = 2, \ldots, N$.

Indeed, for $k = N$, which means only $N$ should be distributed. Meanwhile, $d^{(N)} = N s^{(N)}$. Suppose $s^{(N)} = ps^{'} + q$, where $p \in \mathbb{N}_{0}$ and $q \in \{0, 1, \ldots, s^{'} - 1\}$. Then, 
\[
I_{\max}^{(N)} = N\left\lceil \frac{s^{(N)}}{s^{'}} \right\rceil \quad \text{and} \quad d^{(N)} - s^{(N)} + 1 = (N-1)s^{(N)} + 1.
\]
If $q = 0$, then $p \geq 1$ because $d_{1}=N$, and then we have
\begin{eqnarray*}
    d^{(N)} - s^{(N)} + 1 - I_{\max}^{(N)} &=& (N - 1)p s^{'} + 1 - Np \\
    &\geq& 2(N - 1)p + 1 - Np \\
    &=& (N - 2)p + 1 > 0.
\end{eqnarray*}
If $q \geq 1$, we have
\[
\begin{aligned}
    d^{(N)} - s^{(N)} + 1 - I_{\max}^{(N)} &= (N - 1)(ps^{'} + q) + 1 - N(p + 1) \\
    &\geq 2p(N - 1) - Np + q(N - 1) - N + 1 \\
    &\geq (N - 2)p + (N - 1)(q - 1) \geq 0.
\end{aligned}
\]

Assume that (\ref{aeq6.2.1}) holds for the $(k+1)$-th stage. Now, it is sufficient to prove that it holds for the $k$-th stage with $k \geq 2$, as the claim will then follow by induction.

From the algorithm, it is easy to verify that $I_{\max}^{(k)} \geq I_{\max}^{(k+1)}$. If equality holds, note that 
\[
d^{(k)} - s^{(k)} = (d^{(k+1)} - s^{(k+1)}) + (lk - l) \geq d^{(k+1)} - s^{(k+1)},
\]
where $l = \#\{i : d_{i} = k, \; 1 \leq i \leq s\}$. Then, we have
\[
I_{\max}^{(k)} = I_{\max}^{(k+1)} \leq d^{(k+1)} - s^{(k+1)} + 1 \leq d^{(k)} - s^{(k)} + 1.
\]

If $I_{\max}^{(k)} > I_{\max}^{(k+1)}$, then we claim that $I_{\max}^{(k)} - I_{\min}^{(k)} \leq k$. In fact, in this case, when we are using $k$, there exists a moment such that $\deg R_{j} \leq I_{\max}^{(k+1)}$ for all $j = 1, \ldots, s^{'}$. However, after adding one $k$, we have $\deg R_{j^{*}} > I_{\max}^{(k+1)}$ for some $j^{*} \in \{1, \ldots, s^{'}\}$. Let $I_{\min}^{*}$ denote the minimum degree of $R_{j}$ at this stage, and let $I_{\min}^{*+lk}$ and $I_{\max}^{*+lk}$ denote the minimum and maximum degrees of $R_{j}$, respectively, after adding $lk$ in the stage, where $l \in \mathbb{N}$. Hence,
\begin{eqnarray*}
    I_{\max}^{*+k} = I_{\min}^{*} + k > I_{\max}^{(k+1)} \geq I_{\min}^{*},
\end{eqnarray*}
which implies that $I_{\max}^{*+k} - I_{\min}^{*+k} \leq I_{\max}^{*+k} - I_{\min}^{*} = k$. 

Now, assume that $I_{\max}^{*+lk} - I_{\min}^{*+lk} \leq k$ holds for some $l \in \mathbb{N}$. If there are no additional $k$ values to use, then 
\[
I_{\max}^{(k)} - I_{\min}^{(k)} = I_{\max}^{*+lk} - I_{\min}^{*+lk} \leq k.
\]
If there are still some $k$ values remaining, then 
\[
I_{\max}^{*+(l+1)k} = I_{\min}^{*+lk} + k,
\]
as otherwise it would contradict the assumption that $I_{\max}^{*+lk} - I_{\min}^{*+lk} \leq k$. Therefore,
\[
I_{\max}^{*+(l+1)k} - I_{\min}^{*+(l+1)k} \leq I_{\max}^{*+(l+1)k} - I_{\min}^{*+lk} = k.
\]
Thus, $I_{\max}^{(k)} - I_{\min}^{(k)} \leq k$ holds for all $k \in \mathbb{N}$ in this way.

Hence, we can define $\{\nu + j\}_{j=0}^{k}$ as all the degrees. Assume $t_{j}$ represents the number of $\nu + j$, $j = 0, \ldots, k$, in the $k$-th stage, where $t_{k} \geq 1$ and $\nu \geq k$ (since every $\deg R_{j}$ adds at least $k$, otherwise $I_{\max}^{(k)} = I_{\max}^{(k+1)}$). Then, we have
\begin{eqnarray*}
    I_{\max}^{(k)} = \nu + k, \quad d^{(k)} = \sum_{j=0}^{k} t_{j} (\nu + j), \quad s^{'} = \sum_{j=0}^{k} t_{j}.
\end{eqnarray*}
For any $R_{j}$ in the $k$-th stage, $j = 1, \ldots, s^{'}$, the component is no less than $k$. Hence, there are at most $\frac{\nu + j}{k}$ components in $R_{j}$. Thus,
\begin{eqnarray*}
    s^{(k)} \leq \sum_{j=0}^{k} t_{j} \frac{\nu + j}{k}.
\end{eqnarray*}
Since $k \geq 2$, $t_{k} \geq 1$, and $s^{'} \geq 2$, we have
        \begin{eqnarray*}
            &&d^{(k)}-s^{(k)}+1-I_{\max}^{(k)}\\
            &\geq& \sum_{j=0}^{k}t_{j}(\nu+j)\frac{k-1}{k}+1-\nu-k\\
            &=&s^{'}\nu\frac{k-1}{k}-\nu+\sum_{j=0}^{k}t_{j}j\frac{k-1}{k}+1-k\\
            &\geq&\nu\frac{s^{'}(k-1)-k}{k}+t_{k}k\frac{k-1}{k}+1-k\\
            &\geq&\nu\frac{k-2}{k}\geq 0.
        \end{eqnarray*}

Therefore, the claim (\ref{aeq6.2.1}) follows. Now, we only need to address the first stage ($k = 1$). 

If $I_{\max} = I_{\max}^{(1)} = I_{\max}^{(2)}$, then according to the decreasing properties of $d^{(k)} - s^{(k)} + 1$, we have
\[
I_{\max} \leq d^{(1)} - s^{(1)} + 1 = d - s + 1.
\]

If $I_{\max} = I_{\max}^{(1)} < I_{\max}^{(2)}$, then $I_{\max} - I_{\min} \leq 1$ due to the analysis above, which means there are at most two distinct values, $I_{\max}$ and $I_{\max} - 1$. Let $t \geq 1$ and $s^{'} - t$ denote the numbers of $I_{\max}$ and $I_{\max} - 1$, respectively. Then,
\begin{eqnarray*}
    \left\lceil \frac{d}{s^{'}} \right\rceil = \left\lceil \frac{t I_{\max} + (s^{'} - t)(I_{\max} - 1)}{s^{'}} \right\rceil=\left\lceil I_{\max}-\frac{s^{'}-t}{s^{'}} \right\rceil \geq I_{\max}.
\end{eqnarray*}

Thus, we complete the proof.
    \end{proof}

\noindent\textbf{Remark}: This lemma can also be applied to the classical Second Main Theorem for holomorphic maps $f: \mathbb{C} \to \mathbb{P}^{n}(\mathbb{C})$ with hypersurfaces to obtain more general results. This can be achieved by replacing \cite[Lemma 2.4]{Shi} with the above Lemma \ref{le6.2} in their proof.

\smallskip

\begin{example}\label{ex1}
    Let $Q=\prod_{i=1}^{10}Q_{i}$ be a homogeneous polynomial in $\overline{k}[x_{0}, \ldots, x_{n}]$ where $Q_{i}$ are all homogeneous polynomials in $\overline{k}[x_{0}, \ldots, x_{n}]$ without common factors. Assume $\{\deg Q_{i}\}_{i=1}^{10}=\{6,5, 5, 5, 5, 5, 3, 2, 2, 1\}$. Letting $s^{'}=3$, we proceed to find the decomposition $Q=R_{1}R_{2}R_{3}$ following the algorithm in the proof of Lemma \ref{le6.2}. The following table shows how the algorithm works.
\begin{table}[h]
\caption{The procedure to distribute $Q_{i}$ following the algorithm}\label{tab2}
\begin{tabular*}{\textwidth}{@{\extracolsep\fill}lccccccc}
\toprule%
& \multicolumn{3}{@{}c@{}}{Degree of polynomials} & \multicolumn{4}{@{}c@{}}{Corresponding quantities} \\\cmidrule{2-4}\cmidrule{5-8}%
$k$-th stage & $\deg R_{1}$ & $\deg R_{2}$ & $\deg R_{3}$ & $d^{(k)}$ & $s^{(k)}$ & $I_{\max}^{(k)}$ & $I_{\min}^{(k)}$ \\
\midrule
$6$-th stage  & $6_{1}$ &  &  & 6 & 1 & 6 & 0 \\
$5$-th stage  & $5_{6}$ & $5_{2}+5_{4}$  & $5_{3}+5_{5}$  & 31 & 6 & 11 & 10\\
$4$-th stage  &  &  &  & 31 & 6 & 11 & 10\\
$3$-rd stage  & & $3_{7}$ &  & 34 & 7 & 13 & 10\\
$2$-nd stage  & $2_{9}$ &  & $2_{8}$  & 38 & 9 & 13 & 12\\
$1$-st stage  &  &  & $1_{10}$  & 39 & 10 & 13 & 13\\
\botrule
\end{tabular*}
\footnotetext{Note: The subscript means the order in which $Q_{i}$ is distributed. When we just finished distributing the second $"5"$, we can see that $\deg R_{j}\leq I_{\max}^{(6)}$ for all $j=1, 2, 3$, while $\deg R_{2}>I_{\max}^{(6)}$ after adding one $"5"$. In this case, $I^{*}_{\max}=6$, $I^{*}_{\min}=5$, $I^{*+5}_{\max}=10$ and $I^{*+5}_{\min}=5$.}
\end{table}
\end{example}

\medskip

When we consider the truncated SMT with a decomposition of the polynomials, some relationships between $\nu_{f}^{0}(x)$ and $\nu_{f_{i}}^{0}(x)$ are needed, where $f = \prod f_{i}$. Let us first recall some potential properties of the differential operator used in \cite{Shi}.

Let $M, m \in \mathbb{N}$, and let $f_{1}, \ldots, f_{m}$ be arbitrary entire functions with $f = \prod_{i=1}^{m} f_{i}$. Then we have
\begin{eqnarray}\label{aeq6.111}
    \min\limits_{0\leq t\leq M}\nu_{f^{(t)}}^{0}(x) \geq \sum_{i=1}^{m}\min\limits_{0\leq t\leq M}\nu_{f_{i}^{(t)}}^{0}(x) \geq \min\limits_{0\leq t\leq mM}\nu_{f^{(t)}}^{0}(x)
\end{eqnarray}
for any $x\in\mathbb{C}$. Since $\nu_{f}^{0}(x)=\sum_{i=1}^{m}\nu_{f_{i}}^{0}(x)$, it follows from (\ref{aeq5.6.1.1.1}) that
\begin{eqnarray*}
    \overline{n}^{(mM)}\left(r, \frac{1}{f}\right)\geq \sum_{i=1}^{m} \overline{n}^{(M)}\left(r, \frac{1}{f_{i}}\right)\geq \overline{n}^{(M)}\left(r, \frac{1}{f}\right).
\end{eqnarray*}

For the first inequality of (\ref{aeq6.111}), we only provide the proof when $m=2$ since the other cases can be obtained by a simple induction. Assume $f=f_{1}f_{2}$, then for any $0\leq t\leq M$,
\begin{eqnarray*}
    \nu_{f^{(t)}}^{0}(x)&=&\nu_{\sum_{j=0}^{t}f_{1}^{(j)}f_{2}^{(t-j)}}^{0}(x)\geq \min_{0\leq j\leq t}\nu_{f_{1}^{(j)}}^{0}(x)+\min_{0\leq j\leq t}\nu_{f_{2}^{(j)}}^{0}(x)\\
    &\geq& \min_{0\leq j\leq M}\nu_{f_{1}^{(j)}}^{0}(x)+\min_{0\leq j\leq M}\nu_{f_{2}^{(j)}}^{0}(x).
\end{eqnarray*}
Hence
\begin{eqnarray*}
    \min\limits_{0\leq t\leq M}\nu_{f^{(t)}}^{0}(x) \geq \min_{0\leq t\leq M}\nu_{f_{1}^{(t)}}^{0}(x)+\min_{0\leq t\leq M}\nu_{f_{2}^{(t)}}^{0}(x).
\end{eqnarray*}

For the second inequality, if $\nu_{f}^{0}(x)\leq mM$, then $\min\limits_{0\leq t\leq mM}\nu_{f^{(t)}}^{0}(x)=0$, and hence it is valid. If $\nu_{f}^{0}(x)> mM$, then there exist some $i\in\{1, \cdots, m\}$ such that $\nu_{f_{i}}^{0}(x)> M$. Let $E=\{i:\nu_{f_{i}}^{0}(x)> M\}$ and $p=\# E>0$, then we have $\min\limits_{0\leq t\leq M}\nu_{f_{i}^{(t)}}^{0}(x)=\nu_{f_{i}}^{0}(x)-M>0$ when $i\in E$ and $\sum_{i\in E}\nu_{f_{i}}^{0}(x)\geq \sum_{i=1}^{m}\nu_{f_{i}}^{0}(x)-(m-p)M$. Thus
\begin{eqnarray*}
    \sum_{i=1}^{m}\min\limits_{0\leq t\leq M}\nu_{f_{i}^{(t)}}^{0}(x)&=&\sum_{i\in E} \min\limits_{0\leq t\leq M}\nu_{f_{i}^{(t)}}^{0}(x)=\sum_{i\in E}\left(\nu_{f_{i}}^{0}(x)-M\right)\\
    &\geq& \sum_{i=1}^{m}\nu_{f_{i}}^{0}(x)-mM=\nu_{f}^{0}(x)-mM=\min\limits_{0\leq t\leq mM}\nu_{f^{(t)}}^{0}(x),
\end{eqnarray*}
which proves the second inequality of (\ref{aeq6.111}).

Similar to the first inequality in (\ref{aeq6.111}), we have the following lemma for the operators $\mathcal{A}_{q}$ and $\mathcal{D}_{q}$.

\smallskip

\begin{lemma}\label{le6.4}
    Let $M, m \in \mathbb{N}$ and $f = \prod_{i=1}^{m} f_{i}$, where $f_{1}, \ldots, f_{m}$ are entire functions. Then,
    \begin{eqnarray*}
        \min\limits_{0 \leq t \leq M}\left\{\nu_{\mathcal{A}_{q^{M-t}}\mathcal{D}_{q}^{t}f}^{0}(x)\right\} \geq \sum_{i=1}^{m}\min\limits_{0 \leq t \leq M}\left\{\nu_{\mathcal{A}_{q^{M-t}}\mathcal{D}_{q}^{t}f_{i}}^{0}(x)\right\},
    \end{eqnarray*}
    for all $x \in \mathbb{C}$ satisfying $|x| > \mathcal{R}(M, q)$.
\end{lemma}

\smallskip

\begin{proof}
    By applying Lemma \ref{le5.3}, we have
    \begin{eqnarray*}
        &&\min\limits_{0 \leq t \leq M}\left\{\nu_{\mathcal{A}_{q^{M-t}}\mathcal{D}_{q}^{t}f}^{0}(x)\right\} = \min\limits_{0 \leq t \leq M}\left\{\sum_{i=1}^{m} \nu_{\eta_{q}^{M-2t}f_{i}}^{0}(x)\right\} \\
        &\geq& \sum_{i=1}^{m}\min\limits_{0 \leq t \leq M}\left\{\nu_{\eta_{q}^{M-2t}f_{i}}^{0}(x)\right\} = \sum_{i=1}^{m}\min\limits_{0 \leq t \leq M}\left\{\nu_{\mathcal{A}_{q^{M-t}}\mathcal{D}_{q}^{t}f_{i}}^{0}(x)\right\},
    \end{eqnarray*}
    for all $x \in \mathbb{C}$ satisfying $|x| > \mathcal{R}(M, q)$, which proves the assertion.
\end{proof}

Then we immediately have the following properties for $\tilde{N}_{AW}^{[M]}(r, 1/f)$.

\smallskip

\begin{lemma}\label{le6.4.1}
    Let $M, m\in\mathbb{N}$ and $f=\prod_{i=1}^{m}f_{i}$, where $f_{1}, \dots, f_{m}$ are entire functions. Then,
    \begin{eqnarray*}
        \tilde{N}_{AW}^{[M]}\left(r, \frac{1}{f}\right)\leq \sum_{i=1}^{m}\tilde{N}_{AW}^{[M]}\left(r, \frac{1}{f_{i}}\right)+O(\log r).
    \end{eqnarray*}
\end{lemma}

\begin{proof}
    Since $\nu_{\eta_{q}^{\delta(M)}f}^{0}(x)=\sum_{i=1}^{m}\nu_{\eta_{q}^{\delta(M)}f_{i}}^{0}(x)$ for any $x\in\mathbb{C}$, then the assertion follows according to the definition of (\ref{aeq5.5.2}) by applying Lemma \ref{le6.4}.
\end{proof}

\smallskip

However, there may not be any analogues of the second inequality in (\ref{aeq6.111}) for the operators $\mathcal{A}_{q}$ and $\mathcal{D}_{q}$ due to the discrete property of these operators. In this case, we introduce a new definition of the truncated counting function with a homogeneous polynomial, expressed in terms of 
\[
\sum_{i=1}^{m} \min\limits_{0 \leq t \leq M} \left\{ \nu_{\mathcal{A}_{q^{M-t}} \mathcal{D}_{q}^{t} f_{i}}^{0}(x) \right\},
\]
as follows.

Let a homogeneous polynomial $Q$ of degree $d$ define the hypersurface $D$ in $\mathbb{P}^{n}(\mathbb{C})$, and assume that $Q$ has a min-max $s^{'}$-th decomposition $Q = Q_{1} \cdots Q_{s^{'}}$. Then, we define \textit{the Askey-Wilson version of the zero counting function truncated by $M$ with the $s^{'}$-th decomposition of $Q$ for a holomorphic curve $f$} as

    \begin{eqnarray*}
        \tilde{n}_{AW, f}^{[M](s^{'})}\left(r, Q\right)=\sum\limits_{x\in D(0, r)}\left(\nu_{\eta_{q}^{\delta(M)}Q(\textbf{f})}^{0}(x)-\sum_{i=1}^{s^{'}}\min\limits_{0\leq t\leq M}\left\{\nu_{\mathcal{A}_{q^{M-t}}\mathcal{D}_{q}^{t}Q_{i}(\textbf{f})}^{0}(x)\right\}\right),
    \end{eqnarray*}
where $\delta(M)$ is defined as in (\ref{aeq5.5.3}), and the corresponding integrated counting function is defined by
    \begin{eqnarray*}
        \tilde{N}_{AW, f}^{[M](s^{'})}\left(r, Q\right)&=&\int_{0}^{r}\frac{\tilde{n}_{AW, f}^{[M](s^{'})}(t, Q)-\tilde{n}_{AW, f}^{[M](s^{'})}(0, Q)}{t}dt\\
            &&+\tilde{n}_{AW, f}^{[M](s^{'})}(0, Q)\log r.
    \end{eqnarray*}
It is clear that
\begin{eqnarray}\label{aeq112}
    \tilde{N}_{AW, f}^{[M](1)}(r, Q)=\tilde{N}_{AW}^{[M]}\left(r, \frac{1}{Q(\textbf{f})}\right)
\end{eqnarray}
and both definitions are independent of the choice of the reduced representation of $f$.

The following lemma, due to Quang \cite{Quang}, is frequently used to address the degeneracy SMT.

\smallskip

\begin{lemma}\label{le6.3} (See \cite[Lemma 3.1]{Quang})
    Let $l\in\mathbb{N}$ and $Q_{1}, \ldots, Q_{l+1} \in \mathbb{C}[x_{0}, \ldots, x_{n}]$ be homogeneous polynomials of the same degree $d \geq 1$ such that the hypersurfaces $\{Q_{1} = 0\}, \ldots, \{Q_{l+1} = 0\}$ in $\mathbb{P}^{n}(\mathbb{C})$ are located in $l$-subgeneral position. Then, there exist $n$ homogeneous polynomials $P_{2}, \ldots, P_{n+1}$ of the form
    \begin{eqnarray*}
        P_{t} = \sum_{j=2}^{l-n+t} b_{tj} Q_{j}, \quad b_{tj} \in \mathbb{C}, \quad t = 2, \ldots, n+1,
    \end{eqnarray*}
    such that the hypersurfaces $\{P_{1} = 0\}, \{P_{2} = 0\}, \ldots, \{P_{n+1} = 0\}$ in $\mathbb{P}^{n}(\mathbb{C})$ are in general position, where $P_{1} = Q_{1}$.
\end{lemma}

\smallskip

Now, we are ready to state the main theorem of this paper.

\smallskip

\begin{theorem}[Askey-Wilson version of SMT with hypersurfaces]\label{the6.7}
    Let $f:\mathbb{C}\rightarrow \mathbb{P}^{n}(\mathbb{C})$ be a holomorphic curve with a reduced representation $\textbf{f}=(f_{0}, \cdots, f_{n})$ satisfying that $f_{0}, \cdots, f_{n}$ are algebraically independent, and
    \begin{eqnarray}\label{eq6.5.1}
        \limsup\limits_{r\rightarrow \infty}\frac{\log^{+}\log^{+} T_{f}(r)}{\log\log r}<1.
    \end{eqnarray}
    Let $Q_{j}$ be a homogeneous polynomial in $\mathbb{C}[x_{0}, \cdots, x_{n}]$ of degree $d_{j}(\geq 1)$ defining the hypersurface $D_{j}$ in $\mathbb{P}^{n}(\mathbb{C})$, $j=1, \cdots, p$, such that $D_{1}, \cdots, D_{p}$ are located in $l$-subgeneral position. Suppose $s_{j}\in\mathbb{N}$ is the number of distinct irreducible factors of $Q_{j}$, fix a positive integer $s^{'}\leq \min\left\{n, \min\limits_{1\leq j\leq p}\{s_{j}\}\right\}$, and let $d=\textrm{lcm}\{d_{1}, \cdots, d_{p}\}$.
    Define
    \begin{eqnarray*}
        \alpha=\max\limits_{1\leq j\leq p}\left\{\frac{\max\left\{d_{j}-s_{j}+1, \left\lceil\frac{d_{j}}{s^{'}}\right\rceil\right\}}{d_{j}}\right\}(l-n)
    \end{eqnarray*}
    and
    \begin{eqnarray*}
        \beta=\frac{\alpha d}{\max\limits_{1\leq j\leq p}\left\{\max\left\{d-s_{j}+1, \left\lceil\frac{d}{s^{'}}\right\rceil\right\}\right\}(l-n)}.
    \end{eqnarray*}
    If $\beta(\alpha+1)\geq \alpha$, then, for any $\varepsilon>0$,
    \begin{eqnarray}\label{eq6.5.2}
        (p-(\alpha+1)(n+1)-\varepsilon)T_{f}(r)\leq \frac{1}{d}\sum_{j=1}^{p} \tilde{N}_{AW, f}^{[M_{1}](s^{'})}\left(r, Q_{j}^{\frac{d}{d_{j}}}\right)+S_{\log}(r, f),
    \end{eqnarray}
    where $r$ tends to infinity outside an exceptional set $F$ satisfying (\ref{eq3.3}) and
    \begin{eqnarray*}
        M_{1}=\lfloor 4(e\lceil\alpha+1\rceil(n+1)^{2}d!\lceil \varepsilon^{-1}\rceil)^{n}-1\rfloor.
    \end{eqnarray*}
\end{theorem}

\smallskip

    \begin{proof}
        Let $\tilde{Q}_{j}=Q_{j}^{d/d_{j}}$, $j=1, \cdots, p$. Then $\tilde{Q}_{1}, \cdots, \tilde{Q}_{p}$ have the same degree of $d$. For any given $x\in \mathbb{C}\backslash \bigcup_{j=1}^{p}f^{-1}(D_{j})$, there exists a renumbering $\{1(x), \cdots, p(x)\}$ of $\{1, \cdots, p\}$ such that
        \begin{eqnarray*}
            |\tilde{Q}_{1(x)}(\textbf{f})(x)|\leq\cdots\leq|\tilde{Q}_{p(x)}(\textbf{f})(x)|.
        \end{eqnarray*}
        Since $D_{1}, \cdots, D_{p}$ are in $l$-subgeneral position, we have
        \begin{eqnarray*}
            \lVert \textbf{f}(x)\rVert^{d}\leq A_{1}\cdot \max\limits_{1\leq j\leq l+1}|\tilde{Q}_{j(x)}(\textbf{f})(x)|
        \end{eqnarray*}
        for a finite positive constant $A_{1}$ only depending on the coefficients of $\tilde{Q}_{1}, \dots, \tilde{Q}_{p}$. We denote by $P_{1(x)}=\tilde{Q}_{1(x)}$, $P_{2(x)}$, $\cdots$, $P_{n(x)}$ the homogeneous polynomials obtained in Lemma \ref{le6.3} with respect to the homogeneous polynomials $\tilde{Q}_{1(x)}$, $\cdots$, $\tilde{Q}_{l(x)}$ ($P_{(n+1)(x)}$ and $\tilde{Q}_{(l+1)(x)}$ are ignored here). Then, since there are only finitely many renumberings of $\{1, \cdots, p\}$, we can conclude that there exists a positive constant $A_{2}$, chosen uniformly for all given $x$ such that
               \begin{eqnarray*}
            |P_{t(x)}(\textbf{f})(x)|\leq A_{2}\cdot \max\limits_{2\leq j\leq l-n+t}|\tilde{Q}_{j(x)}(\textbf{f})(x)|=A_{2}\cdot |\tilde{Q}_{(l-n+t)(x)}(\textbf{f})(x)|,
        \end{eqnarray*}
        for $t=2, \cdots, n$.
        Hence
        \begin{eqnarray}\label{aeq6.7.5}
            &&\log\prod_{j=1}^{p}\frac{\left\lVert \textbf{f}(x)\right\rVert^{d}}{|\tilde{Q}_{j(x)}(\textbf{f})(x)|}\leq\log\prod_{j=1}^{l}\frac{\left\lVert \textbf{f}(x)\right\rVert^{d}}{|\tilde{Q}_{j(x)}(\textbf{f})(x)|}+O(1)\nonumber\\
            &=& \log\prod_{j=l-n+2}^{l}\frac{\left\lVert \textbf{f}(x)\right\rVert^{d}}{|\tilde{Q}_{j(x)}(\textbf{f})(x)|}+\log\prod_{j=1}^{l-n+1}\frac{\left\lVert \textbf{f}(x)\right\rVert^{d}}{|\tilde{Q}_{j(x)}(\textbf{f})(x)|}+O(1)\nonumber\\
            &\leq&\log\prod_{j=2}^{n}\frac{\left\lVert \textbf{f}(x)\right\rVert^{d}}{|P_{j(x)}(\textbf{f})(x)|}+\log\prod_{j=1}^{l-n+1}\frac{\left\lVert \textbf{f}(x)\right\rVert^{d}}{|\tilde{Q}_{j(x)}(\textbf{f})(x)|}+O(1)\nonumber\\
            &\leq&\log\prod_{j=1}^{n}\frac{\left\lVert \textbf{f}(x)\right\rVert^{d}}{|P_{j(x)}(\textbf{f})(x)|}+(l-n)\log\frac{\left\lVert \textbf{f}(x)\right\rVert^{d}}{|\tilde{Q}_{1(x)}(\textbf{f})(x)|}+O(1).
        \end{eqnarray}
        Define $Q_{j}=P_{j, 1}^{*}\cdots P_{j, s^{'}}^{*}$ to be the corresponding min-max $s^{'}$-th decomposition for $j=1, \cdots, p$. Let $d_{j, i}^{*}=\deg  P_{j, i}^{*}$, $j=1, \cdots, p$ and $i=1, \cdots, s^{'}$, then $d_{Q_{j}}^{*}(s^{'})=\max\limits_{1\leq i\leq s^{'}}\{d_{j, i}^{*}\}$, and by applying Lemma \ref{le6.2}, we obtain
        \begin{eqnarray}\label{aeq6.7.5.1}
            \frac{d_{Q_{j}}^{*}(s^{'})}{d_{j}}(l-n)\leq \alpha
        \end{eqnarray}
        for $j=1, \cdots, p$. Let
        \begin{eqnarray*}
            \hat{d}&:=&\text{lcm}\{d_{1,1}^{*}, \cdots, d_{1, s^{'}}^{*}, \cdots, d_{p, 1}^{*}, \cdots, d_{p, s^{'}}^{*}, d\}\leq d!\\
            \hat{P}_{j, i}&:=&(P_{j, i}^{*})^{\frac{\hat{d}}{d_{k, i}^{*}}}, \quad j=1, \cdots, p \text{ and } i=1, \cdots, s^{'},\\
        \end{eqnarray*}
        and
        \begin{eqnarray*}
            \hat{P}_{t(x)}&:=&P_{t(x)}^{\frac{\hat{d}}{d}}, \quad t=1, \cdots, n.
        \end{eqnarray*}
        Hence $\deg \hat{P}_{j, i}=\deg \hat{P}_{t(x)}=\hat{d}$. Since $\hat{P}_{j, 1}, \cdots, \hat{P}_{j, s^{'}}$ are pairwise relatively prime, thus we can find another $n-s^{'}$ homogeneous polynomials $\hat{P}_{j, s^{'}+1}, \cdots, \hat{P}_{j, n}$ of degree $\hat{d}$ such that $\hat{P}_{j, 1}, \cdots, \hat{P}_{j, s^{'}}, \hat{P}_{j, s^{'}+1}, \cdots, \hat{P}_{j, n}$  define a subvariety of $\mathbb{P}^{n}(\mathbb{C})$ of dimension $0$. Moreover, there exists a positive constant $A_{3}$ only depending on the coefficients of $Q_{j}$ such that $|\hat{P}_{j, i}(\textbf{f})(x)|\leq A_{3}\lVert \textbf{f}(x)\rVert^{\hat{d}}$ for all $j=1, \cdots, p$ and $i=1, \cdots, n$. Then
        \begin{eqnarray}\label{aeq6.7.6}
            &&(l-n)\log\frac{\left\lVert \textbf{f}(x)\right\rVert^{d}}{|\tilde{Q}_{1(x)}(\textbf{f})(x)|}\nonumber\\
            &=&(l-n)\frac{d}{d_{1(x)}}\log\frac{\left\lVert \textbf{f}(x)\right\rVert^{d_{1(x)}}}{|Q_{1(x)}(\textbf{f})(x)|}\nonumber\\
            &=&(l-n)\frac{d}{d_{1(x)}}\sum_{i=1}^{s^{'}}\log\frac{\left\lVert \textbf{f}(x)\right\rVert^{d_{1(x), i}^{*}}}{|P_{1(x), i}^{*}(\textbf{f})(x)|}\nonumber\\
            &=&(l-n)\frac{d}{d_{1(x)}}\sum_{i=1}^{s^{'}}\frac{d_{1(x), i}^{*}}{\hat{d}}\log \frac{\lVert \textbf{f}(x) \rVert^{\hat{d}}}{|\hat{P}_{1(x), i}(\textbf{f})(x)|}\nonumber\\
            &\leq&(l-n)\frac{d\cdot d_{Q_{1(x)}}^{*}(s^{'})}{d_{1(x)}\cdot \hat{d}}\sum_{i=1}^{s^{'}}\log \frac{\lVert \textbf{f}(x) \rVert^{\hat{d}}}{|\hat{P}_{1(x), i}(\textbf{f})(x)|}+O(1)\nonumber\\
            &\leq&\frac{\alpha d}{\hat{d}}\sum_{i=1}^{n}\log \frac{\lVert \textbf{f}(x) \rVert^{\hat{d}}}{|\hat{P}_{1(x), i}(\textbf{f})(x)|}+O(1),
        \end{eqnarray}
        where the last inequality can be deduced from (\ref{aeq6.7.5.1}). Then it follows by combining (\ref{aeq6.7.5}) and (\ref{aeq6.7.6}) that
        \begin{eqnarray}\label{aeq6.7.7}
            &&\log\prod_{j=1}^{p}\frac{\left\lVert \textbf{f}(x)\right\rVert^{d}}{|\tilde{Q}_{j(x)}(\textbf{f})(x)|}\nonumber\\
            &\leq&\frac{d}{\hat{d}}\sum_{j=1}^{n}\log\frac{\lVert \textbf{f}(x)\rVert^{\hat{d}}}{|\hat{P}_{j(x)}(\textbf{f})(x)|}+\frac{\alpha d}{\hat{d}}\sum_{i=1}^{n}\log \frac{\lVert \textbf{f}(x) \rVert^{\hat{d}}}{|\hat{P}_{1(x), i}(\textbf{f})(x)|}+O(1).
        \end{eqnarray}
        Thus we have two families of homogeneous polynomials $\{\hat{P}_{j(x)}\}_{j=1}^{n}$ and $\{\hat{P}_{1(x), i}\}_{i=1}^{n}$ of degree $\hat{d}$ defining two subvarieties of $\mathbb{P}^{n}$ of dimension $0$ for any fixed $x$.
        
        We define $\mathcal{H}$ as the set of homogeneous polynomial families of common degree $\hat{d}$ constructed by $Q_{1}, \cdots, Q_{p}$ with some (finitely many) specific operations, and, for any element $H=\{H_{1}, \cdots, H_{n}\}\in \mathcal{H}$, $H_{1}, \cdots, H_{n}$ define a subvariety of $\mathbb{P}^{n}(\mathbb{C})$ of dimension $0$. Since there are finitely many hypersurfaces $D_{1}, \cdots, D_{p}$, then $\mathcal{H}$ has only finitely many elements.
   
Next, we discuss the classical filtration method. To maintain the completeness of the proof, we repeat it here. Denote by $V_{N}$ the space of homogeneous polynomials in $\mathbb{C}[x_{0}, \cdots, x_{n}]$ of degree $N$, where $N$ is a fixed large integer which is divisible by $\hat{d}$ to be chosen later. Set 
    \begin{eqnarray}\label{aeq6.7.11}
        M=M_{N}:=\dim V_{N}=\left(\begin{matrix}N+n \\ 
   n
    \end{matrix}\right).
    \end{eqnarray}

For any fixed $\{H_{1}, \cdots, H_{n}\} \in \mathcal{H}$, and according to the lexicographic order, we define a filtration of $V_{N}$ as follows:
        \begin{eqnarray*}
            W(\textbf{i})=\sum_{(\textbf{e})>(\textbf{i})}H_{1}^{e_{1}}\cdots H_{n}^{e_{n}}V_{N-\hat{d}\sigma(\textbf{e})}
        \end{eqnarray*}
        for $n$-tuples $(\textbf{t})=(t_{1}, \cdots, t_{n})\in \mathbb{N}^{n}$ with $\sigma(\textbf{t})=\sum_{s=1}^{n}t_{s}\leq N\slash \hat{d}$ where $(\textbf{t})=(\textbf{i})$ and $(\textbf{e})$. Now we consider the quotient space $W(\textbf{i})\slash W(\textbf{j})$ for some consecutive tuples $(\textbf{i})>(\textbf{j})$ and $\hat{d}\sigma(\textbf{j})<\hat{d}\sigma(\textbf{i})\leq N$, and denote $\Delta(\textbf{i})=\text{dim}W(\textbf{i})\slash W(\textbf{j})$. We choose a suitable basis $\psi_{1}, \cdots, \psi_{M}$ for $V_{N}$ such that $\psi\in\{\psi_{1}, \cdots, \psi_{M}\}$ belongs to the basis of the quotinent space $W(\textbf{i})\slash W(\textbf{j})$ for some consecutive tuples $(\textbf{i})>(\textbf{j})$ and $\hat{d}\sigma(\textbf{j})<\hat{d}\sigma(\textbf{i})\leq N$. 
        
        Since $\sum_{\sigma(\textbf{i})\leq T}=\sum_{\sigma(\textbf{i}^{'})= T}=C_{T+m}^{m}$ for a non-negative integer $T$ where $(\textbf{i})$ and $(\textbf{i}^{'})$ are non-negative integer $m$-tuples and $(m+1)$-tuples respectively, and the sum below is independent of $j$, then by following a lemma due to Corvaja and Zannier in \cite{Cor} (see also \cite[Lemma 3.3]{Ru3} and \cite[Lemma 3.19]{Quang}), we have
        \begin{eqnarray}\label{aeq6.7.12}
            \Omega&=&\Omega_{j}:=\sum_{\sigma(\textbf{i})\leq N/\hat{d}}i_{j}\Delta(\textbf{i})\geq\sum_{\sigma(\textbf{i})\leq N/\hat{d}-n}i_{j}\Delta(\textbf{i})\nonumber\\
            &=&\frac{\hat{d}^{n}}{n+1}\sum_{\sigma(\textbf{i}^{'})=N/\hat{d}-n}\sum_{j=1}^{n+1}i_{j}^{'}=\frac{N(N-\hat{d})\cdots(N-n\hat{d})}{\hat{d}(n+1)!}
        \end{eqnarray}
        and
        \begin{eqnarray}\label{aeq6.7.13}
            \sum_{\sigma(\textbf{i})\leq N/\hat{d}}\sigma(\textbf{i})\Delta(\textbf{i})=n\Omega.
        \end{eqnarray} 
         Now we write $\psi=H_{1}^{i_{1}}\cdots H_{n}^{i_{n}}\kappa$, where $\kappa\in V_{N-\hat{d}\sigma(\textbf{i})}$. Then
        \begin{eqnarray*}
            |\psi(\textbf{f})(x)|\leq A_{4}\cdot|H_{1}(\textbf{f})(x)|^{i_{1}}\cdots |H_{n}(\textbf{f})(x)|^{i_{n}} \lVert \textbf{f}(x)\rVert^{N-\hat{d}\sigma(\textbf{i})},
        \end{eqnarray*}
        where $A_{4}$ is a positive constant depending only on $\psi$. Observe that there are precisely $\Delta(\textbf{i})$ such functions $\psi$ in our basis. Hence, taking the product over all functions in the basis, we obtain, after taking logarithms 
        \begin{eqnarray}\label{aeq6.7.13.1}
            \sum_{t=1}^{M}\log|\psi_{t}(\textbf{f})(x)|&\leq& \sum_{\sigma(\textbf{i})\leq N/\hat{d}}\Delta(\textbf{i})(i_{1}\log|H_{1}(\textbf{f})(x)|+\cdots+i_{n}\log|H_{n}(\textbf{f})(x)|)\nonumber\\
            &&+\left(\sum_{\sigma(\textbf{i})\leq N/\hat{d}}\Delta(\textbf{i})(N-\hat{d}\sigma(\textbf{i}))\right)\log\lVert \textbf{f}(x)\rVert+O(1)\\
            &\leq&\Omega\sum_{j=1}^{n}\log|H_{j}(\textbf{f})(x)|+(NM-\hat{d}n\Omega)\log\lVert \textbf{f}(x)\rVert+O(1).\nonumber
        \end{eqnarray}
        Let $\phi_{1}, \cdots, \phi_{M}$ be a fixed basis of $V_{N}$ and $\textbf{F}=(\phi_{1}(\textbf{f}), \cdots, \phi_{M}(\textbf{f}))$ be a reduced representation of the holomorphic curve $F:\mathbb{C}\rightarrow \mathbb{P}^{M-1}(\mathbb{C})$. Then $\{\psi_{1}, \cdots, \psi_{M}\}$ can be written as linearly independent linear forms $L_{1}, \cdots, L_{M}$ in $\phi_{1}, \cdots, \phi_{M}$ such that $\psi_{t}(\textbf{f})=L_{t}(\textbf{F})$, $t=1, \cdots, M$. Then (\ref{aeq6.7.13.1}) yields
        \begin{eqnarray*}
            \sum_{j=1}^{n}\log\frac{\lVert \textbf{f}(x)\rVert^{|\hat{d}}}{|H_{j}(\textbf{f})(x)|}&\leq&\frac{1}{\Omega}\left(\sum_{t=1}^{M}\log\frac{\lVert \textbf{F}(x)\rVert}{|L_{t}(\textbf{F})(x)|}-M\log\lVert \textbf{F}(x)\rVert\right)\nonumber\\
            &&+\frac{NM}{\Omega}\log\lVert\textbf{f}(x)\rVert+O(1).
        \end{eqnarray*}
        Since $f$ is algebraically independent, then $F$ is linearly independent. From (\ref{aeq6.7.7}), Theorem \ref{the5.1} and the First Main Theorem, we have
        \begin{eqnarray}\label{aeq6.7.13.2}
            &&\int_{0}^{2\pi}\sum_{j=1}^{p}\log \frac{\lVert\textbf{f}(re^{i\theta})\rVert^{d}}{|\tilde{Q}_{j}(\textbf{f})(re^{i\theta})|}\frac{d\theta}{2\pi}\nonumber\\
            &\leq& \frac{(\alpha+1)d}{\hat{d}\Omega}\left(\int_{0}^{2\pi}\max\limits_{\mathcal{K}}\sum_{t\in\mathcal{K}}\log\frac{\lVert \textbf{F}(re^{i\theta}) \rVert}{|L_{t}(\textbf{F})(re^{i\theta})|}\frac{d\theta}{2\pi}-MT_{F}(r)\right)\nonumber\\
            &&+\frac{(\alpha+1)d}{\hat{d}\Omega}NMT_{f}(r)+O(1)\nonumber\\
            &\leq&\frac{(\alpha+1)d}{\hat{d}}\left(\frac{-N_{W}(r, 0)+S_{\log}(r, F)}{\Omega}+\frac{NM}{\Omega}T_{f}(r)\right)+O(1),
        \end{eqnarray}
        where the maximum is taken over all subsets $\mathcal{K}\subset\{1, \cdots, M\}$ such that the linear forms $L_{t}$, $t\in\mathcal{K}$, are linearly independent. Since $T_{F}(r)\leq NT_{f}(r)+O(1)$. Thus, (\ref{aeq6.7.12}) yields $\frac{S_{\log}(r, F)}{\Omega}\leq \frac{N}{\Omega}S_{\log}(r, f)+O(1)=S_{\log}(r, f)$. Then by adding $\sum_{j=1}^{p}N(r, 1/\tilde{Q}_{j}(\textbf{f}))$ to both sides of (\ref{aeq6.7.13.2}) and applying the First Main Theorem, we have
        \begin{eqnarray}\label{aeq6.7.14}
            &&\left(pd-\frac{(\alpha+1)d}{\hat{d}\Omega}NM\right)T_{f}(r)\nonumber\\
            &\leq& \sum_{j=1}^{p}N\left(r, \frac{1}{\tilde{Q}_{j}(\textbf{f})}\right)-\frac{(\alpha+1)d}{\hat{d}\Omega}N_{\mathcal{W}}(r, 0)+S_{\log}(r, f).
        \end{eqnarray}

        Now we first claim that
        \begin{eqnarray}\label{aeq6.7.15}
            &&\sum_{j=1}^{p}N\left(r, \frac{1}{\eta_{q}^{\delta(M_{1})}\tilde{Q}_{j}(\textbf{f})}\right)-\frac{(\alpha+1)d}{\hat{d}\Omega}N_{\mathcal{W}}(r, 0)\nonumber\\
            &\leq& \sum_{j=1}^{p}\tilde{N}_{AW,f}^{[M_{1}](s^{'})}(r, \tilde{Q}_{j})+O(\log r),
        \end{eqnarray}
        where $M_{1}=M-1$. For any fixed $x\in\mathbb{C}$ such that $|x|>\mathcal{R}(M_{1}, q)$, since $D_{1}, \cdots, D_{p}$ are located in $l$-subgeneral position, without loss of generality, we may suppose
        \begin{eqnarray}
            &&\nu^{0}_{\eta_{q}^{\delta(M_{1})}\tilde{Q}_{1}(\textbf{f})}(x)\geq \cdots \geq \nu^{0}_{\eta_{q}^{\delta(M_{1})}\tilde{Q}_{\tau_{1}}(\textbf{f})}(x)>0\nonumber\\
            &=&\nu^{0}_{\eta_{q}^{\delta(M_{1})}\tilde{Q}_{\tau_{1}+1}(\textbf{f})}(x)=\cdots=\nu^{0}_{\eta_{q}^{\delta(M_{1})}\tilde{Q}_{p}(\textbf{f})}(x)\label{aeq110}
        \end{eqnarray}
        where $0\leq \tau_{1}\leq l$. Since $\delta(M_{1})\in\{M_{1}-2t\}_{t=0}^{M_{1}}$, then from Lemma \ref{le5.3}, there exists a injective map $\mu:\{1, \cdots, l\}\rightarrow \{1, \cdots, l\}$ such that
        \begin{eqnarray}
            &&\min\limits_{0\leq t\leq M_{1}}\left\{\nu_{\mathcal{A}_{q^{M_{1}-t}}\mathcal{D}_{q}^{t}\tilde{Q}_{\mu(1)}(\textbf{f})}^{0}(x)\right\}\geq \cdots \geq \min\limits_{0\leq t\leq M_{1}}\left\{\nu_{\mathcal{A}_{q^{M_{1}-t}}\mathcal{D}_{q}^{t}\tilde{Q}_{\mu(\tau_{2})}(\textbf{f})}^{0}(x)\right\}>0\nonumber\\
            &=&\min\limits_{0\leq t\leq M_{1}}\left\{\nu_{\mathcal{A}_{q^{M_{1}-t}}\mathcal{D}_{q}^{t}\tilde{Q}_{\mu(\tau_{2}+1)}(\textbf{f})}^{0}(x)\right\}=\cdots=\min\limits_{0\leq t\leq M_{1}}\left\{\nu_{\mathcal{A}_{q^{M_{1}-t}}\mathcal{D}_{q}^{t}\tilde{Q}_{p}(\textbf{f})}^{0}(x)\right\}\label{aeq111}
        \end{eqnarray}
        where $0\leq \tau_{2}\leq \tau_{1}\leq l$. In particular $\{\mu(1), \cdots, \mu(\tau_{2})\}\subset\{1, \cdots, \tau_{1}\}$. Without loss of generality, we assume $\mu(i)=i$, $i=1, \cdots, l$. Define $\tilde{P}_{1}, \cdots, \tilde{P}_{n}$ be the homogeneous polynomials with degree $d$ constructed by Lemma \ref{le6.3} with respect to $\tilde{Q}_{1}$, $\cdots$, $\tilde{Q}_{l}$ ($\tilde{Q}_{l+1}$). Then
        \begin{eqnarray*}
            \min\limits_{0\leq t\leq M_{1}}\left\{\nu_{\mathcal{A}_{q^{M_{1}-t}}\mathcal{D}_{q}^{t}\tilde{P}_{k}(\textbf{f})}^{0}(x)\right\}&\geq& \min\limits_{0\leq t\leq M_{1}}\min\limits_{2\leq s\leq l-n+k}\left\{\nu_{\mathcal{A}_{q^{M_{1}-t}}\mathcal{D}_{q}^{t}\tilde{Q}_{s}(\textbf{f})}^{0}(x)\right\}\\
            &=&\min\limits_{2\leq s\leq l-n+k}\min\limits_{0\leq t\leq M_{1}}\left\{\nu_{\mathcal{A}_{q^{M_{1}-t}}\mathcal{D}_{q}^{t}\tilde{Q}_{s}(\textbf{f})}^{0}(x)\right\}\\
            &=&\min\limits_{0\leq t\leq M_{1}}\left\{\nu_{\mathcal{A}_{q^{M_{1}-t}}\mathcal{D}_{q}^{t}\tilde{Q}_{l-n+k}(\textbf{f})}^{0}(x)\right\},
        \end{eqnarray*}
        where $k=2, \cdots, n$.
        
        Let $\check{P}_{k} = \tilde{P}_{k}^{\hat{d}/d}$ for $k = 1, \cdots, n$. Then $\{\check{P}_{1}, \cdots, \check{P}_{n}\} \in \mathcal{H}$.
         With respect to $\check{P}_{1}, \cdots, \check{P}_{n}$, by constructing the filtration $W(\textbf{i})$ of $V_{N}$, we can obtain a basis $\{\varphi_{1}, \cdots, \varphi_{M}\}$ of $V_{N}$ and linearly independent linear forms $L_{1}^{'}, \cdots, L_{M}^{'}$, such that $\varphi_{j}(\textbf{f})=L_{j}^{'}(\textbf{F})$. By the property of the Askey-Wilson version of Wronskian-Casorati determinant, we have
        \begin{eqnarray*}
            \mathcal{W}=\mathcal{W}(F_{1}, \cdots, F_{M})=A_{5}\cdot\mathcal{W}(L_{1}^{'}(\textbf{F}), \cdots, L_{M}^{'}(\textbf{F}))=A_{5}\cdot\mathcal{W}(\varphi_{1}(\textbf{f}), \cdots, \varphi_{M}(\textbf{f}))
        \end{eqnarray*}
        where $A_{5}$ is a finite non-zero number, $F_{i}=\phi_{i}(\textbf{f})$ and $\{\phi_{i}\}_{i=1}^{M}$ is a basis of $V_{N}$ which is fixed above. Let $\varphi$ be an element of the basis belonging to $W(\textbf{i})/W(\textbf{j})$ (($\textbf{j}$)$>$($\textbf{i}$) are consecutive $n$-tuples), so we may write $\varphi=\check{P}_{1}^{i_{1}}\cdots \check{P}_{n}^{i_{n}}\kappa_{1}$, $\kappa_{1}\in V_{N-\hat{d}\sigma(\textbf{i})}$. Suppose $\tilde{Q}_{j}=\prod_{i=1}^{s^{'}}R_{j, i}^{*}$ is the min-max $s^{'}$-th decomposition of $\tilde{Q}_{j}$ for $j=1, \cdots, p$. Since the number of distinct irreducible factors of $\tilde{Q}_{j}$ is still $s_{j}$, $j=1, \cdots, p$, then it follows from Lemma \ref{le6.2} that
        \begin{eqnarray*}
            d_{\tilde{Q}_{j}}^{*}(s^{'})\leq \max\left\{d-s_{j}+1, \left\lceil \frac{d}{s^{'}}\right\rceil\right\},
        \end{eqnarray*}
        $j=1, \cdots, p$, and hence
        \begin{eqnarray}\label{aeq6.7.15.1}
            \max\limits_{1\leq j\leq p}d_{\tilde{Q}_{j}}^{*}(s^{'})\leq \frac{\alpha d}{\beta(l-n)}.
        \end{eqnarray}
        Then from Lemma \ref{le5.3} and Lemma \ref{le6.4}, we have
        \begin{eqnarray}\label{aeq6.7.16}
            \nu_{\mathcal{W}}^{0}(x)&\geq& \sum_{s=1}^{M}\min\limits_{0\leq t\leq M_{1}}\left\{\nu_{\mathcal{A}_{q^{M_{1}-t}}\mathcal{D}_{q}^{t}\varphi_{s}(\textbf{f})}^{0}(x)\right\}\nonumber\\
            &\geq& \sum_{(\textbf{i})}\Delta(\textbf{i})\sum_{k=1}^{n}i_{j}\min\limits_{0\leq t\leq M_{1}}\left\{\nu_{\mathcal{A}_{q^{M_{1}-t}}\mathcal{D}_{q}^{t}\check{P}_{k}(\textbf{f})}^{0}(x)\right\}\nonumber\\
            &\geq&\frac{\hat{d}\Omega}{d}\sum_{k=1}^{n}\min\limits_{0\leq t\leq M_{1}}\left\{\nu_{\mathcal{A}_{q^{M_{1}-t}}\mathcal{D}_{q}^{t}\tilde{P}_{k}(\textbf{f})}^{0}(x)\right\}\nonumber\\
            &\geq&\frac{\hat{d}\Omega}{d}\min\limits_{0\leq t\leq M_{1}}\left\{\nu_{\mathcal{A}_{q^{M_{1}-t}}\mathcal{D}_{q}^{t}\tilde{Q}_{1}(\textbf{f})}^{0}(x)\right\}\nonumber\\
            &&+\frac{\hat{d}\Omega}{d}\sum_{k=2}^{n}\min\limits_{0\leq t\leq M_{1}}\left\{\nu_{\mathcal{A}_{q^{M_{1}-t}}\mathcal{D}_{q}^{t}\tilde{Q}_{l-n+k}(\textbf{f})}^{0}(x)\right\}\nonumber\\
            &\geq&\frac{\hat{d}\Omega}{d}\sum_{i=1}^{s^{'}}\min\limits_{0\leq t\leq M_{1}}\left\{\nu_{\mathcal{A}_{q^{M_{1}-t}}\mathcal{D}_{q}^{t}R_{1, i}^{*}(\textbf{f})}^{0}(x)\right\}\nonumber\\
            &&+\frac{\hat{d}\Omega}{d}\sum_{k=2}^{n}\sum_{i=1}^{s^{'}}\min\limits_{0\leq t\leq M_{1}}\left\{\nu_{\mathcal{A}_{q^{M_{1}-t}}\mathcal{D}_{q}^{t}R_{l-n+k, i}^{*}(\textbf{f})}^{0}(x)\right\}.
        \end{eqnarray}

        Among $\tilde{Q}_{2}, \cdots, \tilde{Q}_{l-n+1}$, we choose $j^{*}\in\{2, \cdots l-n+1\}$ such that
        \begin{eqnarray}
            &&\sum_{i=1}^{s^{'}}\min\limits_{0\leq t\leq M_{1}}\left\{\nu_{\mathcal{A}_{q^{M_{1}-t}}\mathcal{D}_{q}^{t}R_{j^{*}, i}^{*}(\textbf{f})}^{0}(x)\right\}\nonumber\\
            &=&\max\limits_{2\leq j \leq l-n+1}\sum_{i=1}^{s^{'}}\min\limits_{0\leq t\leq M_{1}}\left\{\nu_{\mathcal{A}_{q^{M_{1}-t}}\mathcal{D}_{q}^{t}R_{j, i}^{*}(\textbf{f})}^{0}(x)\right\}. \label{aeq6.7.16.1}
        \end{eqnarray}
        Let $\check{P}_{j^{*},i}=(R_{j^{*}, i}^{*})^{\hat{d}/\deg  R_{j^{*}, i}^{*}}$ for $i=1, \cdots, s^{'}$. Then we can find another $n-s^{'}$ homogeneous polynomials of degree $\hat{d}$ such that $\{\check{P}_{j^{*}, 1}, \cdots, \check{P}_{j^{*}, n}\}\in \mathcal{H}$. In a similar way, with respect to $\check{P}_{j^{*}, 1}, \cdots, \check{P}_{j^{*}, n}$, through constructing the filtration $W(\textbf{i})$ of $V_{N}$, we obtain
        \begin{eqnarray}\label{aeq6.7.17}
            \nu_{\mathcal{W}}^{0}(x)&\geq& \Omega\sum_{i=1}^{n}\min\limits_{0\leq t\leq M_{1}}\left\{\nu_{\mathcal{A}_{q^{M_{1}-t}}\mathcal{D}_{q}^{t}\check{P}_{j^{*},i}(\textbf{f})}^{0}(x)\right\}\nonumber\\
            &\geq& \sum_{i=1}^{s^{'}}\frac{\hat{d}\Omega}{\deg R_{j^{*}, i}^{*}}\min\limits_{0\leq t\leq M_{1}}\left\{\nu_{\mathcal{A}_{q^{M_{1}-t}}\mathcal{D}_{q}^{t}R_{j^{*},i}^{*}(\textbf{f})}^{0}(x)\right\}\nonumber\\
            &\geq&\frac{\hat{d}\Omega}{d_{\tilde{Q}_{j^{*}}}^{*}(s^{'})} \sum_{i=1}^{s^{'}}\min\limits_{0\leq t\leq M_{1}}\left\{\nu_{\mathcal{A}_{q^{M_{1}-t}}\mathcal{D}_{q}^{t}R_{j^{*},i}^{*}(\textbf{f})}^{0}(x)\right\}\nonumber\\
            &\geq&\frac{\hat{d}\Omega}{\max\limits_{1\leq j\leq p}d_{\tilde{Q}_{j}}^{*}(s^{'})(l-n)}(l-n)\sum_{i=1}^{s^{'}}\min\limits_{0\leq t\leq M_{1}}\left\{\nu_{\mathcal{A}_{q^{M_{1}-t}}\mathcal{D}_{q}^{t}R_{j^{*},i}^{*}(\textbf{f})}^{0}(x)\right\}\nonumber\\
            &\geq& \beta\frac{\hat{d}\Omega}{\alpha d}\sum_{j=2}^{l-n+1}\sum_{i=1}^{s^{'}}\min\limits_{0\leq t\leq M_{1}}\left\{\nu_{\mathcal{A}_{q^{M_{1}-t}}\mathcal{D}_{q}^{t}R_{j,i}^{*}(\textbf{f})}^{0}(x)\right\},
        \end{eqnarray}
        where the last inequality follows by (\ref{aeq6.7.15.1}) and (\ref{aeq6.7.16.1}). Then, by combining (\ref{aeq110}), (\ref{aeq111}), (\ref{aeq6.7.16}) and (\ref{aeq6.7.17}), we have

        \begin{eqnarray*}
            &&\sum_{j=1}^{p}\nu^{0}_{\eta_{q}^{\delta(M_{1})}\tilde{Q}_{j}(\textbf{f})}(x)-\frac{(\alpha+1)d}{\hat{d}\Omega}\nu_{\mathcal{W}}^{0}(x)\nonumber\\
            &=&\left(\sum_{j=1}^{1}+\sum_{j=l-n+2}^{l}+\sum_{j=2}^{l-n+1}\right)\nu^{0}_{\eta_{q}^{\delta(M_{1})}\tilde{Q}_{j}(\textbf{f})}(x)-\frac{(\alpha+1)d}{\hat{d}\Omega}\nu_{\mathcal{W}}^{0}(x)\nonumber\\
            &\leq&\left(\sum_{j=1}^{1}+\sum_{j=l-n+2}^{l}\right)\left(\nu^{0}_{\eta_{q}^{\delta(M_{1})}\tilde{Q}_{j}(\textbf{f})}(x)-(\alpha+1)\sum_{i=1}^{s^{'}}\min_{0\leq t\leq M_{1}}\left\{\nu^{0}_{\mathcal{A}_{q^{M_{1}-t}}\mathcal{D}_{q}^{t}R^{*}_{j,i}(\textbf{f})}(x)\right\}\right)\nonumber\\
            &&+\sum_{j=2}^{l-n+1}\left(\nu^{0}_{\eta_{q}^{\delta(M_{1})}\tilde{Q}_{j}(\textbf{f})}(x)-\frac{\beta}{\alpha}(\alpha+1)\sum_{i=1}^{s^{'}}\min_{0\leq t\leq M_{1}}\left\{\nu^{0}_{\mathcal{A}_{q^{M_{1}-t}}\mathcal{D}_{q}^{t}R^{*}_{j,i}(\textbf{f})}(x)\right\}\right)\nonumber\\
            &\leq&\sum_{j=1}^{l}\left(\nu^{0}_{\eta_{q}^{\delta(M_{1})}\tilde{Q}_{j}(\textbf{f})}(x)-\sum_{i=1}^{s^{'}}\min_{0\leq t\leq M_{1}}\left\{\nu^{0}_{\mathcal{A}_{q^{M_{1}-t}}\mathcal{D}_{q}^{t}R^{*}_{j,i}(\textbf{f})}(x)\right\}\right)\nonumber
        \end{eqnarray*}
        where the last inequality follows from the assumption that $\beta(\alpha+1)\geq \alpha$. Thus, the claim (\ref{aeq6.7.15}) is obtained.

        Next, we claim that
        \begin{eqnarray}\label{aeq6.7.17.1}
            \sum_{j=1}^{p}\left(N\left(r, \frac{1}{\eta_{q}^{\delta(M_{1})}\tilde{Q}_{j}(\textbf{f})}\right)-N\left(r, \frac{1}{\tilde{Q}_{j}(\textbf{f})}\right)\right)=S_{\log}(r, f).
        \end{eqnarray}
        When $M_{1}$ is even, the left hand side of (\ref{aeq6.7.17.1}) is equal to $0$ because $\delta(M_{1})=0$. When $M_{1}$ is odd, that is $\delta(M_{1})=-1$, we can get the claim by following the proof with Lemma \ref{le3.3} as long as we can get
        \begin{eqnarray}\label{aeq6.7.17.2}
            \limsup_{r\rightarrow \infty}\frac{\log^{+}\log^{+} N(r, 1/\tilde{Q}_{j}(\textbf{f}))}{\log\log r}<1.
        \end{eqnarray}
        
        Indeed, without loss of generality, let $\textbf{g} = (1, f_{1}/f_{0}, \cdots, f_{n}/f_{0})$. Since $f_{0}, \cdots, f_{n}$ are algebraically independent, it follows that $\tilde{Q}_{j}(\textbf{f}) = f_{0}^{d} \tilde{Q}_{j}(\textbf{g}) \not\equiv 0$. By applying Lemma \ref{le2.4}, we have
        \begin{eqnarray*}
            N\left(r, \frac{1}{\tilde{Q}_{j}(\textbf{f})}\right)&\leq&dN\left(r, \frac{1}{f_{0}}\right)+N\left(r, \frac{1}{\tilde{Q}_{j}(\textbf{g})}\right)\\
            &\leq&dT_{f}(r)+T(r, \tilde{Q}_{j}(\textbf{g}))+O(1)\leq A_{6}\cdot T_{f}(r),
        \end{eqnarray*}
        where $j=1, \cdots, p$ and $A_{6}$ is a finite positive number only depending on $\tilde{Q}_{j}$, which implies (\ref{aeq6.7.17.2}) according to the assumption (\ref{eq6.5.1}), and hence (\ref{aeq6.7.17.1}) holds. Then by combining (\ref{aeq6.7.14}), (\ref{aeq6.7.15}) and (\ref{aeq6.7.17.1}), we have
        \begin{eqnarray}\label{aeq6.7.17.3}
            \left(pd-\frac{(\alpha+1)d}{\hat{d}\Omega}NM\right)T_{f}(r)\leq \sum_{j=1}^{p}\tilde{N}_{AW,f}^{[M_{1}](s^{'})}(r, \tilde{Q}_{j})+S_{\log}(r, f).
        \end{eqnarray}

Lastly, we should determine the large number $N$ according to $\varepsilon$. Since $\frac{y}{1+y}\leq \log(1+y)\leq y$ for all $y\geq 0$, then when $y\in \left[0, \frac{1}{n(n+1)}\right]$, we have
        \begin{eqnarray*}
            \log(1+y)^{n}-\log(1+(n+1)y)\leq ny-\frac{(n+1)y}{1+(n+1)y}=\frac{n(n+1)y-1}{1+(n+1)y}y\leq 0
        \end{eqnarray*}
        for $n\in\mathbb{N}$, which implies 
        \begin{eqnarray}\label{aeq6.7.18}
            (1+y)^{n}\leq 1+(n+1)y, \quad y\in\left[0, \frac{1}{n(n+1)}\right].
        \end{eqnarray}
     Now, let $N = \lceil \alpha + 1 \rceil (n+1)^{3} \hat{d} \lceil \varepsilon^{-1} \rceil + (n+1)\hat{d}$. Then $N$ is divisible by $\hat{d}$, and
        \begin{eqnarray*}
            0\leq \frac{(n+1)\hat{d}}{N-(n+1)\hat{d}}=\frac{1}{\lceil\alpha+1\rceil(n+1)^{2}\lceil \varepsilon^{-1}\rceil}\leq \frac{1}{n(n+1)}.
        \end{eqnarray*}
        It is easy to check that $\frac{N+j}{N-(n+1-j)\hat{d}}\leq 1+\frac{(n+1)\hat{d}}{N-(n+1)\hat{d}}$ for all $j=1, \cdots, n$. Then, from (\ref{aeq6.7.11}), (\ref{aeq6.7.12}), (\ref{aeq6.7.13}) and (\ref{aeq6.7.18}), we have
        \begin{eqnarray}\label{aeq6.7.19}
            \frac{NM}{\hat{d}\Omega}&\leq&(n+1)\frac{(N+1)\cdots (N+n)}{(N-n\hat{d})\cdots (N-\hat{d})}\nonumber\\
            &\leq&(n+1)\left(1+\frac{(n+1)\hat{d}}{N-(n+1)\hat{d}}\right)^{n}\nonumber\\
            &\leq&(n+1)+\frac{(n+1)^{3}\hat{d}}{N-(n+1)\hat{d}}\nonumber\\
            &\leq&(n+1)+\frac{\varepsilon}{\alpha+1}
        \end{eqnarray}
        and 
        \begin{eqnarray}\label{aeq6.7.20}
            M&\leq& e^{n}\left(1+\frac{N}{n}\right)^{n}\nonumber\\
            &\leq&e^{n}\left(\frac{n+(n+1)\hat{d}}{n}+\frac{\lceil\alpha+1\rceil(n+1)^{3}\hat{d}\lceil \varepsilon^{-1}\rceil}{n}\right)^{n}\nonumber\\
            &\leq&(e\lceil\alpha+1\rceil(n+1)^{2}d!\lceil \varepsilon^{-1}\rceil)^{n}\left(1+\frac{1}{n}+\frac{n+(n+1)\hat{d}}{\lceil\alpha+1\rceil n(n+1)^{2}\hat{d}\lceil \varepsilon^{-1}\rceil}\right)^{n}\nonumber\\
            &\leq&(e\lceil\alpha+1\rceil(n+1)^{2}d!\lceil \varepsilon^{-1}\rceil)^{n}\left(1+\frac{1}{n}+\frac{2}{n(n+1)}\right)^{n}\nonumber\\
            &\leq&4(e\lceil\alpha+1\rceil(n+1)^{2}d!\lceil \varepsilon^{-1}\rceil)^{n},
        \end{eqnarray}

where we used the basic estimation 
\[
\left(1+\frac{n+3}{n(n+1)}\right)^{n} = \left(1+\frac{n+3}{n(n+1)}\right)^{\frac{n(n+1)}{n+3} \cdot \frac{n+3}{n+1}} \leq 4
\]
for all $n \in \mathbb{N}$.  Indeed, it clearly holds when $n=1$. For $n\geq 2$, we have $\frac{n(n+1)}{n+3}\geq 1$ and $\frac{n+3}{n+1}\leq 2$. 
Since $h(y) = \left(1+\frac{1}{y}\right)^{y}$ is decreasing for $y \geq 1$, the estimation follows.
Then, the assertion follows by combining (\ref{aeq6.7.17.3}), (\ref{aeq6.7.19}), and (\ref{aeq6.7.20}).
\end{proof}

In particular, when $s^{'} = 1$ in the above theorem, we have the 
normal Askey-Wilson analogue of truncated Second Main Theorem as follows. In this case, $\alpha = l - n$ and $\beta = 1$ satisfy $\beta(\alpha+1)\geq \alpha$.

\smallskip

\begin{corollary}
Let $f:\mathbb{C} \rightarrow \mathbb{P}^{n}(\mathbb{C})$ be a holomorphic curve with a reduced representation $\textbf{f} = (f_{0}, \cdots, f_{n})$ satisfying that $f_{0}, \cdots, f_{n}$ are algebraically independent, and
\begin{eqnarray}\label{eq6.5.1.1}
    \limsup\limits_{r \to \infty} \frac{\log^{+} \log^{+} T_{f}(r)}{\log \log r} < 1.
\end{eqnarray}
Let $Q_{j}$ be a homogeneous polynomial in $\mathbb{C}[x_{0}, \cdots, x_{n}]$ of degree $d_{j}(\geq 1)$, defining the hypersurface $D_{j}$ in $\mathbb{P}^{n}(\mathbb{C})$, $j=1, \cdots, p$, such that $D_{1}, \cdots, D_{p}$ are located in $l$-subgeneral position. Let $d=\textrm{lcm}\{d_{1}, \cdots, d_{p}\}$. Then, for any $\varepsilon > 0$,
\begin{eqnarray}\label{eq6.5.2.1}
    (p - (l - n + 1)(n + 1) - \varepsilon)T_{f}(r) \leq \sum_{j=1}^{p} \frac{1}{d_{j}} \tilde{N}_{AW}^{[M_{1}]}\left(r, \frac{1}{Q_{j}(\textbf{f})}\right) + S_{\log}(r, f),
\end{eqnarray}
where $r$ tends to infinity outside an exceptional set $F$ satisfying (\ref{eq3.3}), and
\begin{eqnarray*}
    M_{1} = \lfloor 4(e \lceil l - n + 1 \rceil (n + 1)^{2} d! \lceil \varepsilon^{-1} \rceil)^{n} - 1 \rfloor.
\end{eqnarray*}
\end{corollary}
\begin{proof}
    By applying (\ref{aeq112}) and Theorem \ref{the6.7}, we have the right hand side of (\ref{eq6.5.2.1}) should be
    \begin{eqnarray*}
        \frac{1}{d} \sum_{j=1}^{p} \tilde{N}_{AW}^{[M_{1}]}\left(r, \frac{1}{Q_{j}^{\frac{d}{d_{j}}}(\textbf{f})}\right) + S_{\log}(r, f).
    \end{eqnarray*}
    Then the assertion follows from Lemma \ref{le6.4.1}.
\end{proof}
    \smallskip

\bmhead*{Acknowledgements} We are thankful to Professor Tohge for suggesting useful references that enhanced the quality of this research. The first author would like to thank the support of the China Scholarship Council (Grant No. 202206020078).


\bibliography{sn-bibliography}%
\end{document}